\documentclass[12pt]{iopart}
\usepackage{}
\usepackage{graphicx}
\usepackage{algorithm}
\usepackage{caption}
\usepackage{subcaption}
\usepackage{epstopdf}
\AppendGraphicsExtensions{.tif}
\usepackage{iopams}
\usepackage{amsthm, amsfonts,amssymb}
\usepackage{placeins}
\usepackage{enumitem}
\usepackage{cleveref}
\usepackage{textcomp}
\usepackage{multicol}
\graphicspath{ {./Figures/} }

\newtheorem{thm}{Theorem}[section]
\newtheorem{Lem}[thm]{Lemma}

\newtheorem{prop}[thm]{Proposition}

\newtheorem{rem}{Remark}

\theoremstyle{definition}
\newtheorem{defn}{Definition}[section]
\theoremstyle{remark}

\usepackage{iopams}  
\begin{document}

\title[Renormalization of symmetric bimodal maps with low smoothness]{Renormalization of symmetric bimodal maps with low smoothness}

\author{Rohit Kumar$^1$, V.V.M.S. Chandramouli$^2$ }

\address{$^{1,2}$ Department of Mathematics, Indian Institute of Technology Jodhpur, \;\;\; Rajasthan, India-342037. }
\eads{ $^{1}$  \mailto kumar.30@iitj.ac.in, $^{2}$  \mailto chsarma@iitj.ac.in.}
\vspace{10pt}
 \date{\today}

\begin{abstract}
\ This paper deals with the renormalization of symmetric bimodal maps with low smoothness. We prove the existence of 
the renormalization fixed point in the space $C ^{1+Lip}$ symmetric bimodal maps. Moreover, we show that the topological entropy of the renormalization operator defined on the space of $C^{1+Lip}$ symmetric bimodal maps is infinite. Further we prove the existence of a continuum of fixed points of renormalization. Consequently, this proves the non-rigidity of the renormalization of symmetric bimodal maps.  
\end{abstract}

%
\vspace{2pc}
\eqnobysec

\section{Introduction}
Renormalization is a technique to analyze maps having the property that the first return map to small part of the phase space resembles the original map itself. Period doubling renormalization operator was introduced by M. Feigenbaum \cite{Fe}, \cite{Fe2} and by P. Coullet and C. Tresser \cite{CT}, to study asymptotic small scale geometry of the attractor of one dimensional systems which are at the transition from simple to chaotic dynamics.

With a relatively complete understanding of the period doubling renormalization of unimodal maps, recent research in dynamical systems has focused on more complicated maps of the real line. Renormalization is very useful tool to describe the dynamics of a map in smaller scale. In particular, Jonker \& Rand \cite{JD} and V. Strien \cite{SVST} used renormalization as a natural vehicle to decompose the non-wandering set in a hierarchical manner, for unimodal maps. 
The multimodal maps are interesting as generalizations of unimodal maps, as well as for their applications. For example, in the case of bimodal maps, they are essential to the understanding of non-invertible circle maps which have been used extensively to model the transitions to chaos in two frequency systems \cite{MCT}.  D. Veitch presented some work on topological renormalization of $C^0$ bimodal maps with zero and positive entropy \cite{DVeitch}. Further, D. Smania developed a combinatorial theory for certain kind of multimodal maps and proved that for the same combinatorial type the renormalizatons of infinitely renormalizable smooth multimodal maps are exponentially close \cite{DSmania}. 

In this paper, we focus on the construction of renormalization fixed point for the family of symmetric bimodal maps with low smoothness. 
We show that there exists a sequence of affine pieces which are nested and contract to the critical points of the bimodal map corresponding to a pair of proper scaling data.
This helps us to prove that the renormalization operator defined on the space of piece-wise affine infinitely renormalizable maps has a fixed point, denoted by $f_{s^*},$ corresponding to a pair of proper scaling data $s^*$. In the next section~\ref{extsn}, we explain the extension of the renormalization fixed point $f_{s^*}$ to a $C^{1+Lip}$ symmetric bimodal map. In section~\ref{entropy}, we describe the topological entropy of renormalization defined on the space of $C^{1+Lip}$ symmetric bimodal maps. Furthermore, we prove the existence of another fixed point of renormalization by considering the small perturbation on the scaling data.  Consequently, for two different perturbed scaling data we get two Cantor attractors of renormalization fixed points. This leads to the non-rigidity of the Cantor attractors of renormalizable symmetric bimodal map with low smoothness.  \\

\noindent We recall some basic definitions.  Let $I = [a,b]$  be a closed interval. \\

 A unimodal map $\mathfrak{u}: I \rightarrow I $ is called \textit{period tripling renormalizable map} if there exists a proper subinterval $J$ of $I$ such that \\
(1)\; $J,$ $\mathfrak{u}(J)$ and $\mathfrak{u}^2(J)$ are pairwise disjoint, \\
(2)\;  $\mathfrak{u}^3(J) \subset J.$ \\
\noindent Then $\mathfrak{u}^3: J \rightarrow J$ is called a renormalization of $\mathfrak{u}.$ \\

A map $\mathfrak{u}: I \rightarrow I$ is \textit{period tripling infinitely renormalizable} map if there exists an infinite sequence $\{I_n\}_{n = 0}^\infty$ of nested intervals such that ${\mathfrak{u}^{3}|_{I_n}} : I_n \rightarrow I_n$ are renormalizations of $\mathfrak{u}$ and the length of $I_n$ tends to zero as $n \rightarrow \infty.$ \\

\noindent Let $\mathcal{U}$ be the collection of unimodal maps and $\mathcal{U}_\infty ( \subset \mathcal{U}) $ be the collection of period tripling infinitely renormalizable unimodal maps. \\

\ An interval map $f$ is \textit{piece-wise monotone} if there exists a partition of $I$ into finitely many subintervals  on each of which $f$ is strictly monotonic.\\ If three is the minimal number of such subintervals, we say $f$ is a \textit{bimodal} map.

\begin{defn}
Let $f : I \rightarrow I$ be a map with two subsets $J_l$ and $J_r$ such that  $J_l\strut^\mathrm{o} \cap J_r^{\strut\mathrm{o}} = \emptyset.$ If $f|_{J_l}$ and $f|_{J_r}$ are unimodal maps which are concave up and concave down respectively, their \textit{join}, denoted by $f|_{J_l} \oplus f|_{J_r}$, is a bimodal map whose graph is obtained by joining $f(max(J_l))$  and $f(min(J_r))$ by a $C^{1+Lip}$ curve.
\end{defn}

\begin{defn}
A bimodal map $b : I \rightarrow I ,$ having two critical points $c_l$ and $c_r,$ is said to be \textit{renormalizable} if there exists two disjoint intervals $I_l$ containing $c_l$ and  $I_r$ containing $c_r$ such that 
\begin{enumerate}
\item[(i)] $b^i(I_l) \cap b^j(I_l) = \emptyset,$ for each $i \neq j$ and $i,j \in \{0,1,2\},$ \\
$b^i(I_r) \cap b^j(I_r) = \emptyset,$ for each $i \neq j$ and $i,j \in \{0,1,2\},$ 
\item[(ii)] $b^3(I_l) \subset I_l$ and $b^3(I_r) \subset I_r,$
\item[(iii)] After applying suitable reflections and rescalings on the unimodal maps $b|_{I_l}$ and $b|_{I_r},$ the unimodal pieces $b'_{l}$ and $b'_{r}$ are joined to generate a bimodal map.
\end{enumerate}
\end{defn}

The renormalization of a symmetric bimodal map is illustrated in Figure \ref{fig:brenom}.

\begin{figure}[h]
	\begin{center}
		\includegraphics[scale=0.6]{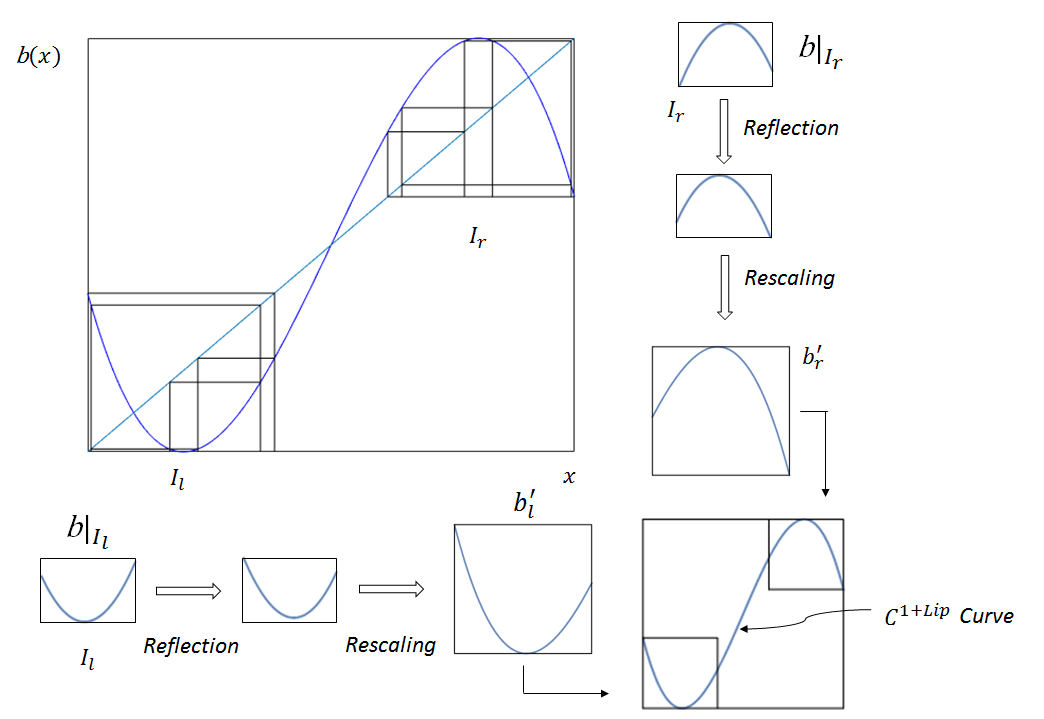}
	\end{center}
\caption{Renormalization of a bimodal map}
\label{fig:brenom}
\end{figure}

In the next section, we construct the renormalization operator defined on the space of piece-wise affine maps which are infinitely renormalizable maps. 

\section{Piece-wise affine renormalizable maps}\label{p2}
A symmetric bimodal map $b: [0,1] \rightarrow [0,1]$ of the form $b(x) = a_3x^3+a_2 x^2+a_1x+a_0,$ for $a_3<0,$ is a $C^1$ map with the following conditions
\begin{itemize}
	\item $b(0)= 1-b(1),$
	\item $b(\frac{1}{2})= \frac{1}{2},$
	\item let $c_l$  and $c_r$ be the two critical points of $b(x)$ , then 
	$b(c_l) = 0$ and $b(c_r)=1.$
\end{itemize} 

\ Let us consider a family of symmetric bimodal maps $B_c : [0, 1] \rightarrow [0, 1]$ which are increasing on the interval between the critical points and decreasing elsewhere. then, we obtained a family of bimodal maps as
\begin{eqnarray}
B_c(x)& = \left\{\begin{array}{ll}
	1-\frac{1 - 6 c + 9 c^2 - 4 c^3 + 6 c x - 6 c^2 x - 3 x^2 + 2 x^3}{(1 - 2 c)^3}, \;\;\; \;  \; \textrm{if} \; c \in \left[0, \frac{1}{4}\right]  \\ 1-\frac{4c^3 - 3 c^2 + 6 c x - 6 c^2 x - 3 x^2 + 2 x^3}{( 2 c-1)^3}, \;\;\;  \; \;\;\;\;\;\;\; \; \textrm{if} \; c \in \left[\frac{3}{4}, 1\right]
\end{array}
\right. \nonumber \\	
& \equiv \left\{\begin{array}{ll}
	b_{c}(x), \;\;\; \;  \; \textrm{if} \; c \in \left[0, \frac{1}{4}\right]  \\ \tilde{b}_{c}(x), \;\;\;  \;  \; \textrm{if} \; c \in \left[\frac{3}{4}, 1\right]
\end{array}
\right. 	\label{eqb}
\end{eqnarray}

\noindent Note that the bimodal maps $b_c$ and $\tilde{b}_c$ are identical maps. 

\noindent Let us define an open set 
$$\Delta^3   = \left\{ (s_0,s_1,s_2) \in \mathbb{R}^3 \; : \;  s_0,s_1,s_2 > 0, \;  \sum\limits_{i = 0}^{2} s_i < 1 \right\}. $$
Each element $(s_0,s_1,s_2)$ of $\Delta^3$ is called a scaling tri-factor. A pair of scaling tri-factors $(s_{0,l},s_{1,l},s_{2,l})$ and $(s_{0,r},s_{1,r},s_{2,r})$ induces two sets of affine maps $(F_{0,l},F_{1,l},F_{2,l})$ and $(F_{0,r},F_{1,r},F_{2,r})$ respectively. For each $i = 0,1,2,$
\begin{eqnarray*}
 F_{i,l}  : I_L  = [0 , b_c(0)] \rightarrow I_L \hspace{1.5cm}\textrm{and}\hspace{1.5cm}   F_{i,r}  : I_R = [ \tilde{b}_c(1), 1] \rightarrow I_R 
\end{eqnarray*} 
 are defined as 
\begin{eqnarray*}
 F_{0,l}(t) = b_c(0) - s_{0,l} \cdot t; \qquad&  F_{0,r}(t) = \tilde{b}_c(1) + s_{0,r} \cdot (1-t) \\ 
 F_{1,l}(t) =  b_c^2(0) - s_{1,l} \cdot t;  & F_{1,r}(t) =  \tilde{b}_c^2(1) + s_{1,r} \cdot (1-t) \\
 F_{2,l}(t) = s_{2,l} \cdot t; & F_{2,r}(t) = 1-s_{2,r} \cdot (1-t).
\end{eqnarray*}

\noindent Note that ${I_L}^{\strut\mathrm{o}} \cap {I_R}^{\strut\mathrm{o}} = \phi,$ for  $c \in [0,\frac{3-\sqrt{3}}{6}].$
\\

The functions $ s_l : \mathbb{N} \rightarrow \Delta^3 $ and $ s_r : \mathbb{N} \rightarrow \Delta^3 $ are said to be a scaling data. We set scaling tri-factors \\ $ s_l(n) = (s_{0,l}(n),s_{1,l}(n), s_{2,l}(n)) \in \Delta^3 $  and $ s_r(n) = (s_{0,r}(n),s_{1,r}(n), s_{2,r}(n)) \in \Delta^3,$\\ so that $s_l(n)$ and $s_r(n)$ induce the triplets of affine maps $ (F_0^l(n)(t) , F_1^l(n)(t) , F_2^l(n)(t))$ and $ (F_0^r(n)(t) , F_1^r(n)(t) , F_2^r(n)(t))$ as described above.\\
For $i=0,1,2,$ let us define the intervals
\begin{eqnarray*}
I_{i,l}^{n} = F_{1,l}(1)\circ F_{1,l}(2)\circ F_{1,l}(3) \circ.....\circ F_{1,l}(n-1)\circ F_{i,l}(n)([0,b_c(0)]).
\end{eqnarray*}
Also,
\begin{eqnarray*}
I_{i,r}^{n} = F_{1,r}(1)\circ F_{1,r}(2)\circ F_{1,r}(3) \circ.....\circ F_{1,r}(n-1)\circ F_{i,r}(n)([\tilde{b}_c(1),1]).
\end{eqnarray*}

\begin{defn}
\ A scaling data $s(n)$ is said to be proper if $$ d(s(n),\partial \Delta^3) \geq \epsilon, \;\;\;\;\;   \textrm{for some} \;\;\; \epsilon > 0.$$ 
\end{defn} 
    
\noindent  A pair of proper scaling data $ s_l : \mathbb{N} \rightarrow \Delta^3 $ and $ s_r : \mathbb{N} \rightarrow \Delta^3 $ induce the sets $ D_{s_l}  = \bigcup\limits_{n \geq 1} (I_{0,l}^{n} \cup I_{2,l}^{n})$ and $ D_{s_r}  = \bigcup\limits_{n \geq 1} (I_{0,r}^{n} \cup I_{2,r}^{n}),$ respectively.   Consider a map $$ f_{s} : D_{s_l}  \cup D_{s_r} \rightarrow [0,1] $$  defined as        $$f_s(x) = \left\{\begin{array}{ll}
f_{s_l}(x), \;\;\; \;  \; \textrm{if} \; x \in D_{s_l}  \\ f_{s_r}(x), \;\;\;  \;  \; \textrm{if} \; x \in D_{s_r}
\end{array}
\right.$$ where $f_{s_l}|_{I_{0,l}^{n}}$ and $f_{s_l}|_{I_{2,l}^{n}}$ are the affine extensions of  $b_c|_{\partial I_{0,l}^{n}}$ and $b_c|_{\partial I_{2,l}^{n}}$ respectively. Similarly, $f_{s_r}|_{I_{0,r}^{n}}$ and $f_{s_r}|_{I_{2,r}^{n}}$ are the affine extensions of  $b_c|_{\partial I_{0,r}^{n}}$ and $b_c|_{\partial I_{2,r}^{n}}$ respectively. These affine extensions are shown in Figure \ref{fig:bibdry}.

\FloatBarrier
\begin{figure}[h]
	\centering
	{\includegraphics [width=100mm]{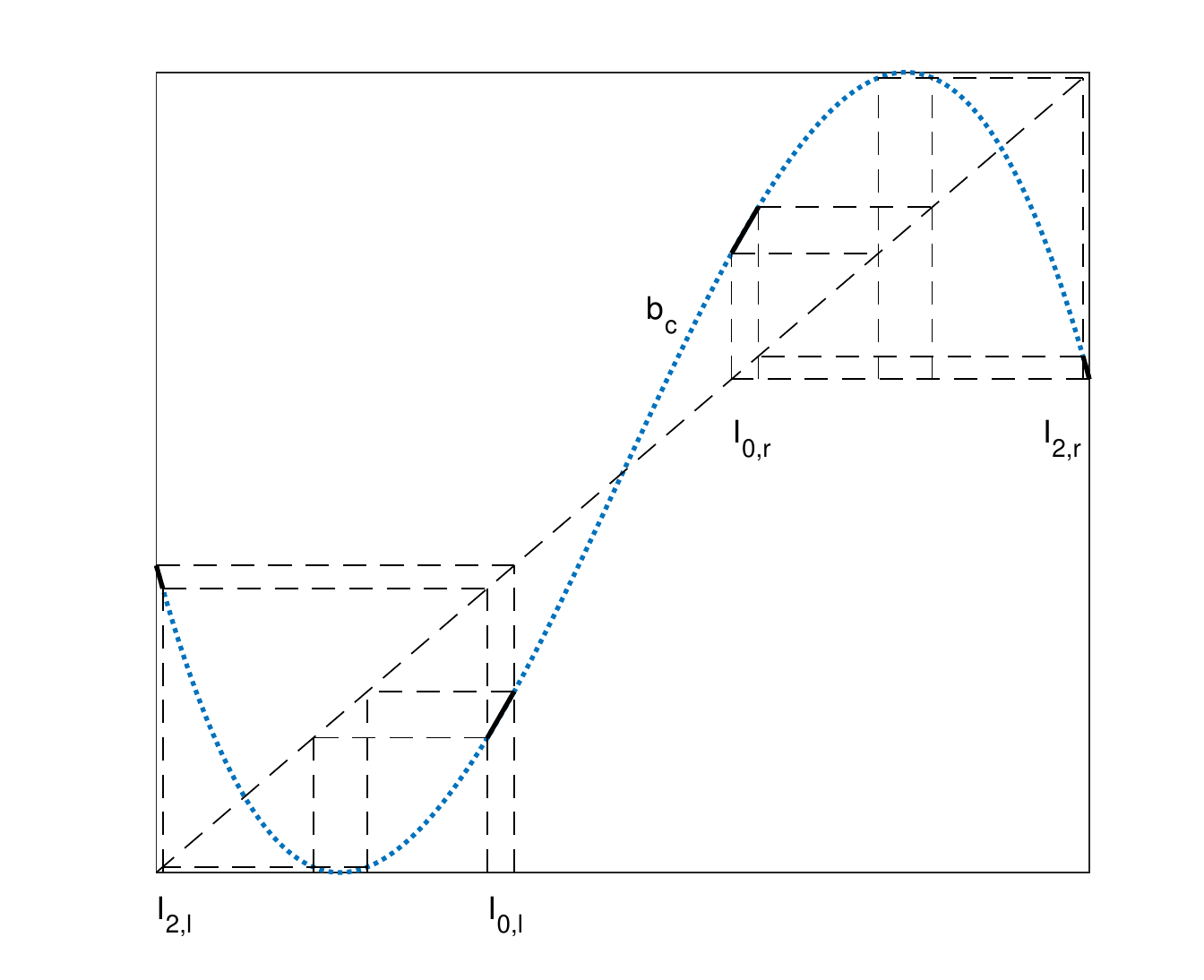}}
	\caption{}
	\label{fig:bibdry}
\end{figure} 
\FloatBarrier

\noindent The end points of the intervals at each level are labeled by
 
\noindent  $$y_0 = 0,\; z_0 = b_c(0),\; I_{1,l}^0 = I_L = [0,b_c(0)] $$ and for $ n \geq 1$ 
\begin{center}
$ x_n =  \partial I_{0,l}^n \backslash \partial I_{1,l}^{n-1} $\\
$ y_{2n-1} = max\{ \partial I_{1,l}^{2n-1} \} $\\
$ y_{2n} = min\{ \partial I_{1,l}^{2n} \} $\\
$ z_{2n-1} = min\{ \partial I_{1,l}^{2n-1} \} $\\
$ z_{2n} = max\{ \partial I_{1,l}^{2n} \} $\\
$ w_n =  \partial I_{2,l}^n \backslash \partial I_{1,l}^{n-1}. $
\end{center}

These points are illustrated in Figure~\ref{fig:nexgen}.
\FloatBarrier
\begin{figure}[h]
	\centering
	{\includegraphics [width=100mm]{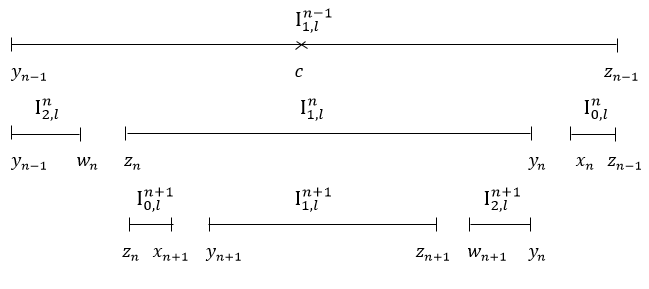}}
	\caption{Intervals of next generations}
	\label{fig:nexgen}
\end{figure} 
\FloatBarrier
\noindent Also, the end points of the intervals at each level are labeled by
 
\noindent  $$z'_0 = b _c(1), \; y'_0 = 1, \; I_{1,r}^0 = I_R = [\tilde{b}_c(1),1] $$ and for $ n \geq 1$ 
\begin{center}
$ x'_n =  \partial I_{0,r}^n \backslash \partial I_{1,r}^{n-1} $\\
$ y'_{2n-1} = min\{ \partial I_{1,r}^{2n-1} \} $\\
$ y'_{2n} = max\{ \partial I_{1,r}^{2n} \} $\\
$ z'_{2n-1} = max\{ \partial I_{1,r}^{2n-1} \} $\\
$ z'_{2n} = min\{ \partial I_{1,r}^{2n} \} $\\
$ w'_n =  \partial I_{2,r}^n \backslash \partial I_{1,r}^{n-1}. $
\end{center}

These points are illustrated in Figure~\ref{fig:nexgenr}.
\FloatBarrier
\begin{figure}[h]
	\centering
	{\includegraphics [width=100mm]{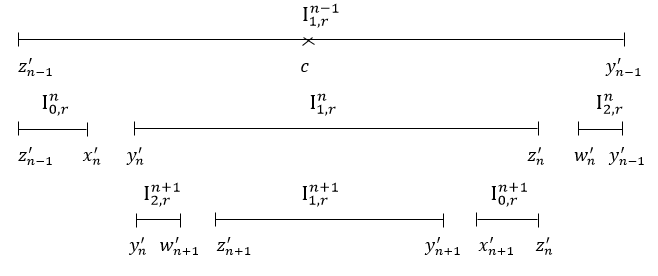}}
	\caption{Intervals of next generations}
	\label{fig:nexgenr}
\end{figure} 
\FloatBarrier

\begin{defn}
\ For a given pair of proper scaling data $s_l, s_r : \mathbb{N} \rightarrow \Delta^3, $ a map $f_{s}$ is said to be \textit{infinitely renormalizable} if for $n \geq 1, $ 
\begin{itemize}
\item[(1)]  $[0,f_{s_l}(y_{n})]$ is the maximal domain containing $0$ on which $f_{s_l}^{3^n-1}$ is defined affinely, $[f_{s_l}^2(y_{n}), f_{s_l}(0)]$ is the maximal domain containing $f_{s_l}(0)$ on which $f_{s_l}^{3^n-2}$ is defined affinely,
$[f_{s_r}(y'_{n}),1]$ is the maximal domain containing $1$ on which $f_{s_r}^{3^n-1}$ is defined affinely and $[ f_{s_r}(1),f_{s_r}^2(y'_{n})]$ is the maximal domain containing $f_{s_r}(1)$ on which $f_{s_r}^{3^n-2}$ is defined affinely,
\item[(2)] $f_{s_l}^{3^n-1}([0,f_{s_l}(y_{n})]) = I_{1,l}^n,$ \\ $f_{s_l}^{3^n-2}([f_{s_l}^2(y_{n}), f_{s_l}(0)]) = I_{1,l}^n $ \\  $f_{s_r}^{3^n-1}([f_{s_r}(y'_{n}),1]) = I_{1,r}^n,$ \\ $f_{s_r}^{3^n-2}([f_{s_r}(1),f_{s_r}^2(y'_{n})]) = I_{1,r}^n.$
\end{itemize}
\end{defn}

\noindent The combinatorics for renormalization of $f_{s_l}$ and $f_{s_r}$ are shown in the following Figures \ref{fig:ptr} and \ref{fig:ptr1}.

\begin{figure}[!htb]
 \centering
  \begin{subfigure}[b]{0.4\linewidth}
    \centering\includegraphics[width=170pt]{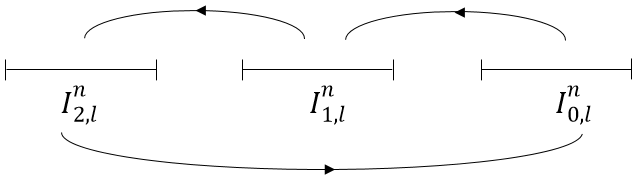}
    \caption{\label{fig:ptr}}
  \end{subfigure}%
  \begin{subfigure}[b]{0.4\linewidth}
    \centering\includegraphics[width=170pt]{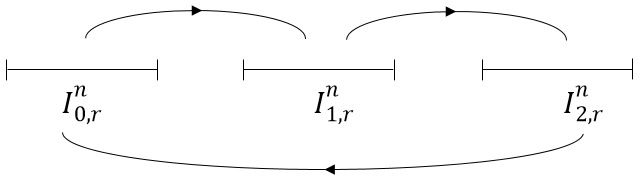}
    \caption{\label{fig:ptr1}}
  \end{subfigure} 
  \caption{The combinatorics: (\subref{fig:ptr}) corresponding to $f_{s_l},$ ($ I_{1,l}^n \rightarrow I_{2,l}^n \rightarrow I_{0,l}^n \rightarrow I_{1,l}^n)$ and  (\subref{fig:ptr1}) corresponding to $f_{s_r},$  ($ I_{1,r}^n \rightarrow I_{2,r}^n \rightarrow I_{0,r}^n \rightarrow I_{1,r}^n).$  }
\end{figure}
 
\subsection{Renormalization on $I_L = [0,b_c(0)]$}\label{p21}
 \noindent Let $f_{s_l} \in \mathcal{U}_\infty$ be given by the proper scaling data $s_l : \mathbb{N} \rightarrow \Delta^3 $ and define $$ {\tilde{I}}_{1,l}^n  = [max\{b_c^{-1}( z_{n})\},b_c(0)] = [ max\{f_{s_l}^{-1}( z_{n})\},f_{s_l}(0)] ,$$  where $ b_c^{-1}(x)$ denotes the preimage(s) of $x$ under $b_c$ and $$ {\hat{I}}_{1,l}^n  = [0,b_c(y_{n})] = [0,f_{s,l}(y_{n})].$$  Let $$ h_{s_l,n} : [0,b_c(0)] \rightarrow [0,b_c(0)] $$ be defined by $$h_{s_l,n} = F_{1,l}(1)\circ F_{1,l}(2)\circ F_{1,l}(3) \circ.....\circ F_{1,l}(n)$$
Furthermore, let $${\tilde{h}}_{s_l,n} : [0,b_c(0)] \rightarrow {\tilde{I}}_{1,l}^n \;\; \textrm{and} \;\; {\hat{h}}_{s_l,n} : [0,b_c(0)] \rightarrow {\hat{I}}_{1,l}^n $$ be the affine orientation preserving homeomorphisms. Then define $$R_n^lf_s: h_{s_l,n}^{-1}(D_{s_l}) \rightarrow [0,b_c(0)]$$  by $$R_n^lf_{s_l}(x) =  \left\{\begin{array}{ll}
R_n^{l-}f_{s_l}(x), \;\;\; \;  \; \textrm{if} \; x \in h_{s_l,n}^{-1}(I_{0,l}^n) \\ R_n^{l+}f_{s_l}(x), \;\;\;  \;  \; \textrm{if} \; x \in h_{s_l,n}^{-1}(I_{2,l}^n)
\end{array}
\right.$$ \\
where, $$R_n^{l-}f_{s_l}: h_{s_l,n}^{-1}(\mathop{\cup}\limits_{n \geq 1}I_{0,l}^n) \rightarrow [0,b_c(0)] $$\;\; and $$ R_n^{l+}f_{s_l}: h_{s_l,n}^{-1}(\mathop{\cup}\limits_{n \geq 1}I_{2,l}^n) \rightarrow [0,b_c(0)]$$ are defined by $$R_n^{l-}f_{s_l}(x) = {\tilde{h}}_{s_l,n}^{-1} \circ f_{s_l}^{-1} \circ h_{s_l,n}(x)$$ $$R_n^{l+}f_{s_l}(x) = {\hat{h}}_{s_l,n}^{-1} \circ f_{s_l} \circ h_{s_l,n}(x), $$ which are illustrated in Figure~\ref{fig:renomop}.

\begin{figure}[!htb]
\centering
{\includegraphics [width=120mm]{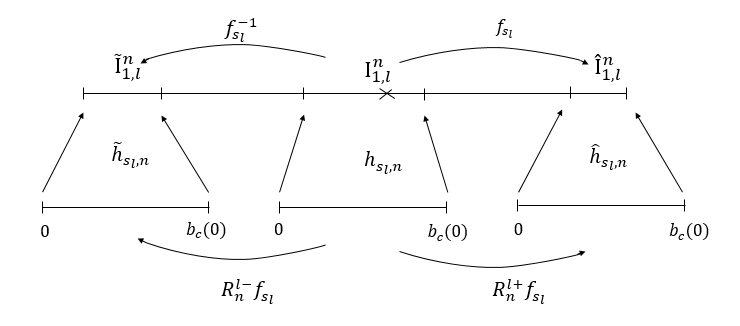}}
\caption{ }
 \label{fig:renomop}
\end{figure}

\noindent Let $\sigma: (\Delta^3)^{\mathbb{N}} \rightarrow (\Delta^3)^{\mathbb{N}}$ be the shift map defined as $ \sigma(s_l^1s_l^2s_l^3s_l^4....) = (s_l^2s_l^3s_l^4....), $ where $s_l^i \in \Delta^3$ for all $i \in \mathbb{N}.$

\begin{Lem}\label{lem1}
\ Let $s_l : \mathbb{N} \rightarrow \Delta^3$ be proper scaling data such that $f_{s_l}$ is infinitely renormalizable. Then $$R_n^l{f_{s_l}} = f_{\sigma^{n}(s_l)}. $$
\end{Lem}

\noindent Let $f_{s_l}$ be infinitely renormalization, then for $n \geq 0,$ we have $$ f_{s_l}^{3^n}: D_{s_l} \cap I_{1,l}^n \rightarrow I_{1,l}^n $$ is well defined. \\
\ Define the renormalization $R^l : \mathcal{U}_\infty \rightarrow \mathcal{U}_\infty$ by
$$ R^lf_{s_l} = h_{s_l,1}^{-1} \circ f_{s_l}^3 \circ h_{s_l,1} .$$

\noindent The maps $f_{s_l}^{3^n-2} : {\tilde{I}}_{1,l}^n \rightarrow I_{1,l}^n $ and $f_{s_l}^{3^n-1} : {\hat{I}}_{1,l}^n \rightarrow I_{1,l}^n $ are the affine homeomorphisms whenever $f_{s_l} \in \mathcal{U}_\infty$. Then

\begin{Lem}\label{lem2}
\ We have $(R^l)^n{f_{s_l}} : D_{\sigma^{n}(s_l)} \rightarrow [0,b_c(0)]$ and $(R^l)^n{f_{s_l}} = R_n^l{f_{s_l}}.$
\end{Lem}

\noindent The~\cref{lem1} and~\cref{lem2}  give the following result.

\begin{prop}\label{pl}
\ There exists a map $f_{s_l^*} \in  \mathcal{U}_\infty,$ where $s_l^*$ is characterized by $$R^lf_{s_l^*} = f_{s_l^*}. $$ 
\end{prop}
\begin{proof}
\ Consider $s_l : \mathbb{N} \rightarrow \Delta^3$ be proper scaling data such that $f_{s_l}  $ is an infinitely renormalizable. Let $c_n$ be the critical point of $f_{\sigma^n(s_l)}.$ Then

\FloatBarrier
\begin{figure}[h]
\centering
{\includegraphics [scale=0.65]{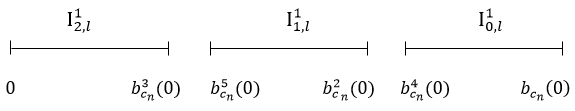}}
\caption{Length of intervals}
\label{fig:ilength}
\end{figure} 
\FloatBarrier

\noindent we have the following scaling ratios which are illustrated in Figure~\ref{fig:ilength}
\begin{eqnarray}
 s_{0,l}(n) &=   \frac{b_{c_n}(0)-b_{c_n}^4(0)}{b_{c_n}(0)} \label{eq01} \\ 
  s_{1,l}(n) &=  \frac{b_{c_n}^2(0)-b_{c_n}^5(0)}{b_{c_n}(0)} \label{eq02} \\
  s_{2,l}(n) &=  \frac{b_{c_n}^3(0)}{b_{c_n}(0)} \label{eq03} \\
 c_{n+1} &= \frac{b_{c_n}^2(0)-c_n}{s_{1,l}(n)}  \equiv \mathcal{R}(c_n).  \label{eq04}
\end{eqnarray}

\noindent Since $(s_{0,l}(n), s_{1,l}(n), s_{2,l}(n)) \in \Delta^3,$ this implies the following conditions 
\begin{eqnarray}
 s_{0,l}(n),\; s_{1,l}(n), \; s_{2,l}(n) &> 0 \label{eq05} \\
 s_{0,l}(n) +s_{1,l}(n) +s_{2,l}(n) & < 1 \label{eq06} 
 \end{eqnarray}

\ As the intervals $I_{i,l}^n,$ for $i = 0,1,2,$ are mutually disjoint, we denote the gap ratios as $g_{0,l}^n$ and $g_{1,l}^n$ which are in between $I_{0,l}^n \;\&\; I_{1,l}^n$ and $I_{1,l}^n \; \& \; I_{2,l}^n$ respectively. The gap ratios are defined as,  \\
for $n \in \mathbb{N},$
\begin{eqnarray}
g_{0,l}^n = \frac{b_{c_n}^4(0)-b_{c_n}^2(0)}{b_{c_n}(0)} \equiv G_{0,l}(c_n) >0  \label{g1} \\
g_{1,l}^n = \frac{b_{c_n}^5(0)-b_{c_n}^3(0)}{b_{c_n}(0)} \equiv G_{1,l}(c_n) > 0 \label{g2} \\
0 < c_n < \frac{3-\sqrt{3}}{6} \label{eq6a}
\end{eqnarray}

\noindent We use Mathematica for solving the equations~(\ref{eq01}),~(\ref{eq02}) and~(\ref{eq03}), then we get the expressions for $s_{0,l}(n),\;s_{1,l}(n)$ and $\;s_{2,l}(n).$ \\ Let $s_{i,l}(n) \equiv S_{i,l}(c_n)$ for $i = 0,1,2.$  The graphs of $S_{i,l}(c)$ are shown in Figures ~\ref{fig:fig1}, ~\ref{fig:fig2} and~\ref{fig:fig3}. 

\begin{figure}[ht]
  \centering
  \begin{subfigure}[b]{0.4\linewidth}
    \centering\includegraphics[width=135pt]{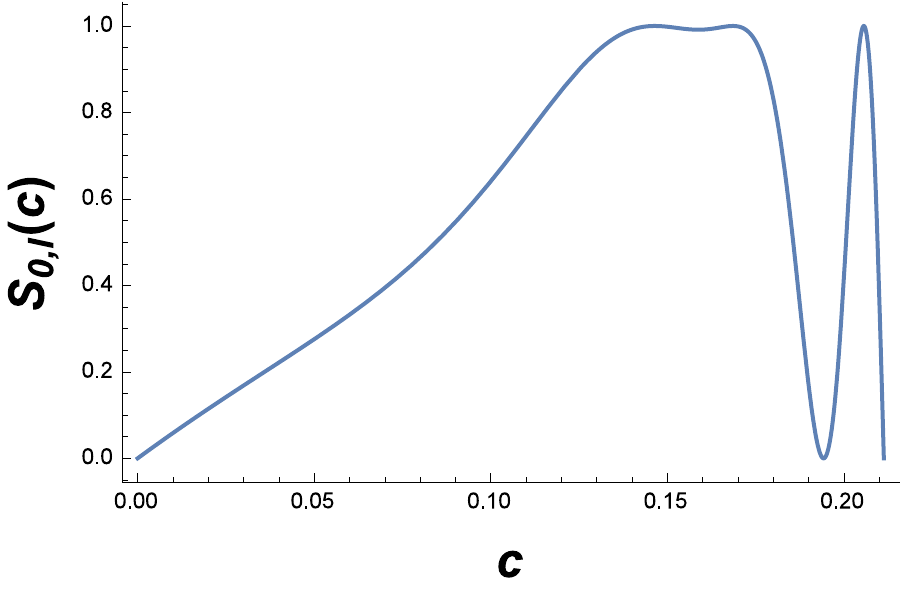}
    \caption{\label{fig:fig1}}
  \end{subfigure}%
  \begin{subfigure}[b]{0.4\linewidth}
    \centering\includegraphics[width=135pt]{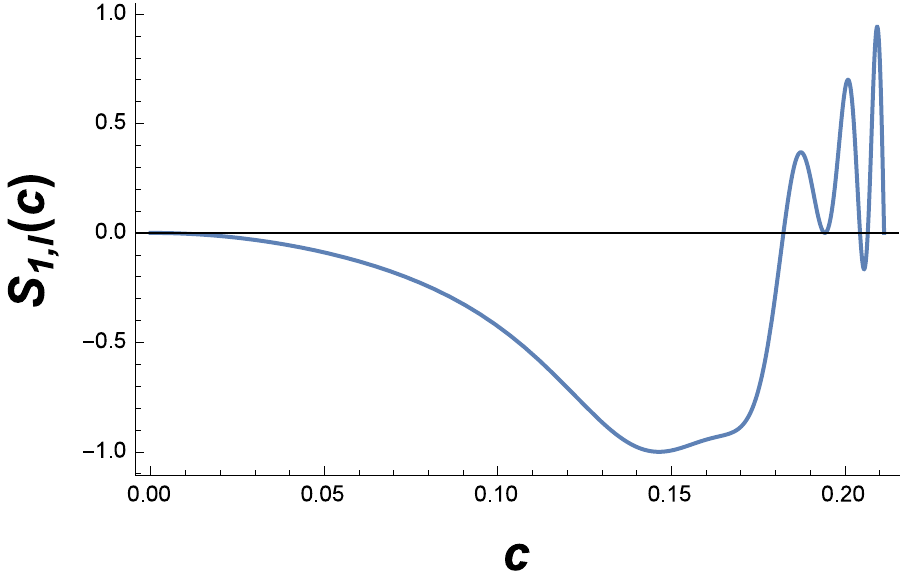}
    \caption{\label{fig:fig2}}
  \end{subfigure}
  \begin{subfigure}[b]{0.4\linewidth}
    \centering\includegraphics[width=135pt]{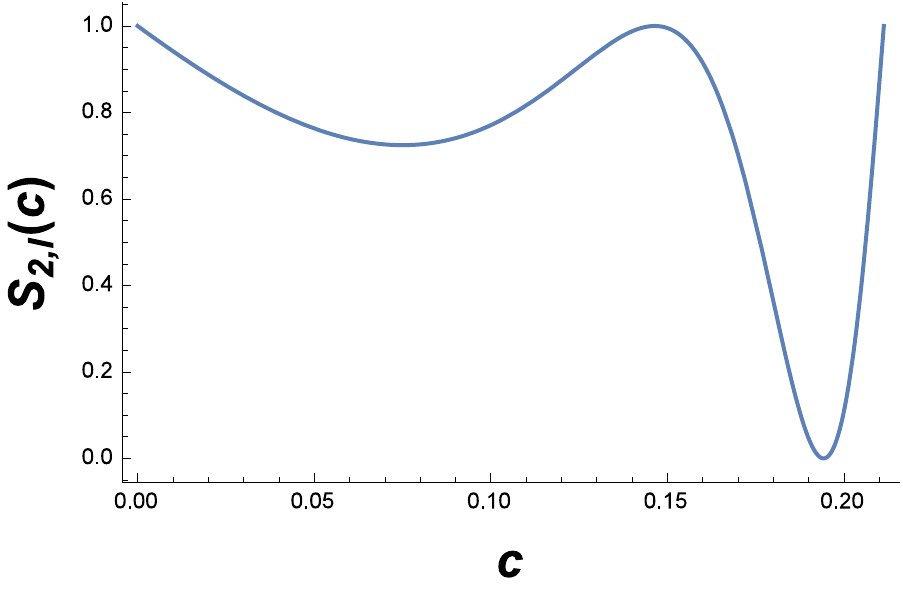}
    \caption{\label{fig:fig3}}
  \end{subfigure}%
  \begin{subfigure}[b]{0.4\linewidth}
    \centering\includegraphics[width=135pt]{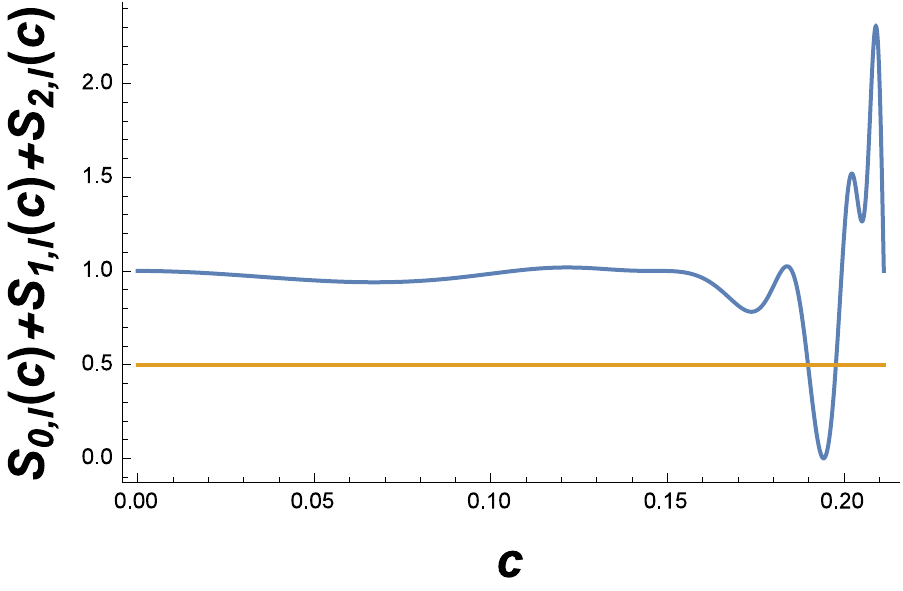}
    \caption{\label{fig:fig4}}
  \end{subfigure}
  \caption{  (\subref{fig:fig1}), (\subref{fig:fig2}), (\subref{fig:fig3}) and (\subref{fig:fig4}) show the graphs of $S_{0,l}(c),$  $S_{1,l}(c),$ $S_{2,l}(c)$ and \\ $(S_{0,l}+S_{1,l}+S_{2,l})(c)$.}
\label{fig:sratio}
\end{figure}
\newpage
\noindent Note that the conditions (\ref{eq05}), (\ref{g1}) and (\ref{g2}) give the condition (\ref{eq06}) $$ 0<\sum_{i=0}^{2} s_{i,l}(n) < 1.$$

The conditions ~(\ref{eq05}) together with ~(\ref{g1}) to ~(\ref{eq6a}) define the feasible domain $F_d^l$ is to be: 
\begin{eqnarray}\label{eqfd}
\fl F_d^l = \Big\{ \; c \in \left(0,\; \frac{3-\sqrt{3}}{6}\right) \; : \;  S_{i,l}(c) > 0 \; \textrm{for} \; i=0,1,2, & G_{0,l}(c) >0,  G_{1,l}(c) >0  \Big\} .
\end{eqnarray} 

To compute the feasible domain $F_d^l,$  we need to find subinterval(s) of $\left(0, \frac{3-\sqrt{3}}{6}\right)$ which satisfies the conditions of (\ref{eqfd}). By using Mathematica software, we employ the following command to obtain the feasible domain
\begin{eqnarray*}
\fl \textup{N[Reduce}[\bigg\{S_{0,l}(c) > 0, & S_{1,l}(c) > 0, S_{2,l}(c) > 0,  G_{0,l}(c) > 0,  G_{1,l}(c) > 0,  0<c<\frac{3-\sqrt{3}}{6} \bigg\}, c]].
\end{eqnarray*}
This yields: 
\begin{eqnarray*}
 F_d^l = (0.188816...,\; 0.194271...) \cup (0.194271...,\;  0.199413...) \equiv F_{d_1}^l \cup F_{d_2}^l.
\end{eqnarray*}

From the Eqn.(\ref{eq04}), the graphs of $\mathcal{R}(c)$ are plotted in the sub-domains $F_{d_1}^l$ and $F_{d_2}^l$ of $F_d^l$ which are shown in Figure \ref{fig:R}.

\begin{figure}[htp]
	\centering
	\begin{subfigure}[b]{0.5\linewidth}
		\centering\includegraphics[width=160pt]{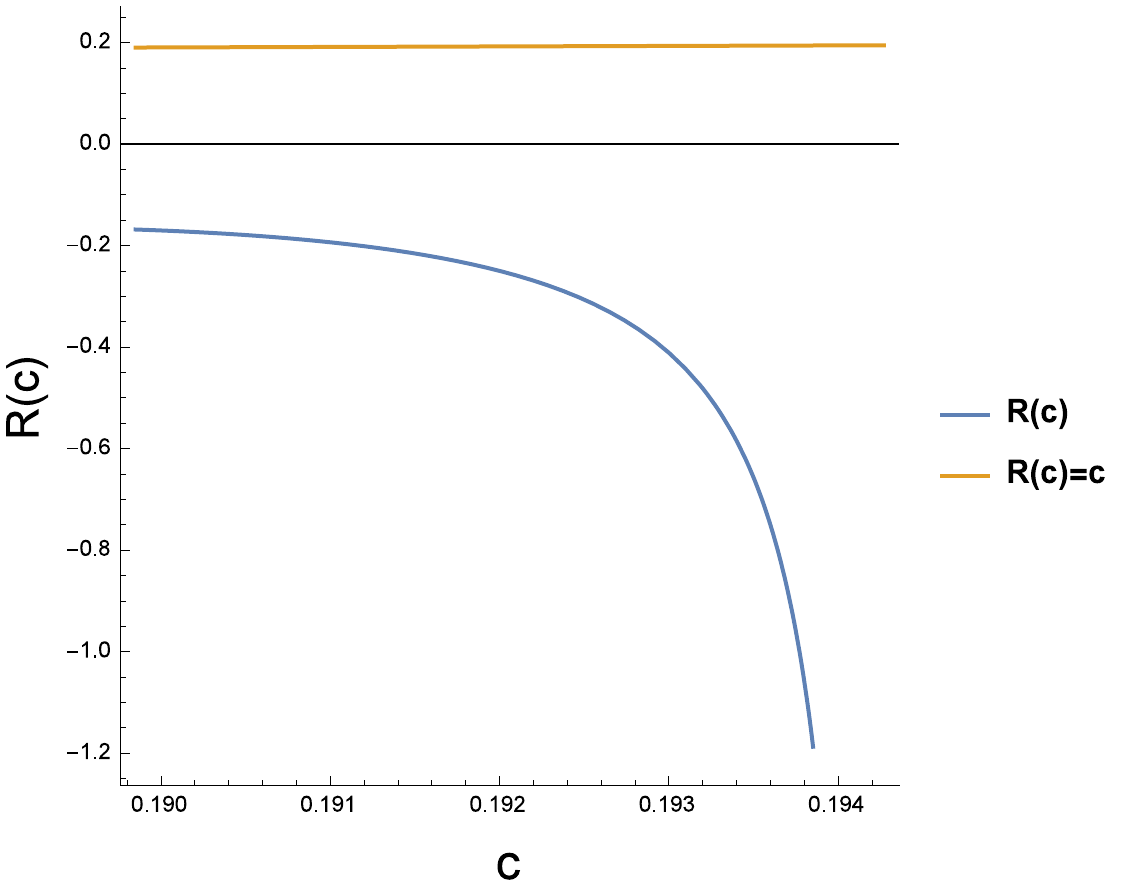}
		\subcaption{ $\mathcal{R}$ has no fixed point in $F_{d_1}^l.$ }
		\label{fig:picf1}
	\end{subfigure}%
	\begin{subfigure}[b]{0.5\linewidth}
		\centering\includegraphics[width=160pt]{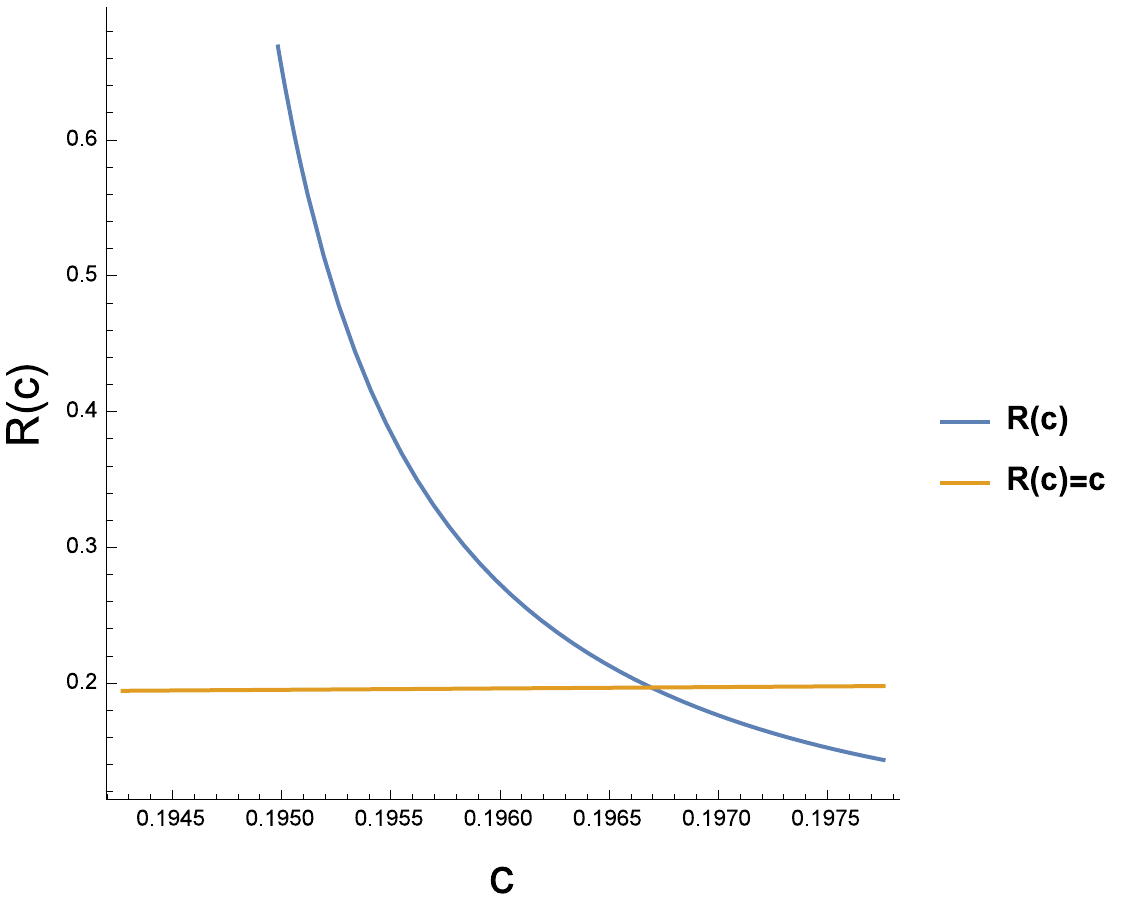}
		\subcaption{$\mathcal{R}$ has only one fixed point in $F_{d_2}^l.$ }
		\label{fig:picf2}
	\end{subfigure}
	\caption{The graph of $\mathcal{R} : F_d^l \rightarrow \mathbb{R}$ and the diagonal $\mathcal{R}(c) = c.$}
	\label{fig:R}
\end{figure} 

The map $\mathcal{R}  : F_d^l \rightarrow \mathbb{R}$ is expanding in the neighborhood of fixed point $c_l^*$ which is illustrated in Figure~\ref{fig:picf2}. By Mathematica computations, we get an unstable fixed points $c_l^* = 0.196693...$ in $F_d^l $ such that  $$\mathcal{R}(c_l^*) = c_l^*$$ corresponds to an infinitely renormalizable maps $f_{s_l^*}.$ We observe that the map $ f_{{s_l}^*}$ corresponding to $c_l^*$ has the following property $$\{c_l^*\} = \bigcap\limits_{n \geq 1} I_{1,l}^n. $$

\noindent In other words, consider the scaling data ${s_l}^* : \mathbb{N} \rightarrow \Delta^3$  with 
 \begin{eqnarray*}
 {s_l}^*(n) &= (s_{0,l}^*(n), s_{1,l}^*(n), s_{2,l}^*(n) )  \\
  &=\left(\frac{b_{c_l^*}(0)-b_{c_l^*}^4(0)}{b_{c_l^*}(0)}, \; \frac{b_{c_l^*}^2(0)-b_{c_l^*}^5(0)}{b_{c_l^*}(0)} , \; \frac{b_{c_l^*}^3(0)}{b_{c_l^*}(0)}\right).
 \end{eqnarray*}
  Then $\sigma({s_l^*} ) = s_l^* $ and using Lemma \ref{lem1} we have $$R^lf_{s_l^*} = f_{s_l^*}.$$

\end{proof}

\subsection{Renormalization on $I_R = [\tilde{b}_c(1), 1]$}\label{p22}

\noindent In subsection \ref{p21}, the bimodal map $b_c(x)$ has two critical points $c \in I_L$ and $1-c \in I_R$ and we define the piece-wise renormalization on $I_L.$ In similar fashion, to define the renormalization on $I_R$ with $c \in I_R$, from Equation \ref{eqb}, we consider $$ \tilde{b}_c(x) = 1-\frac{4c^3 - 3 c^2 + 6 c x - 6 c^2 x - 3 x^2 + 2 x^3}{( 2 c-1)^3}$$ where $x \in [0,1]$ and $c  \in [ \frac{3}{4},1].$

\noindent Note that ${I_L}^{\strut\mathrm{o}} \cap {I_R}^{\strut\mathrm{o}} = \phi,$ for  $c \in [\frac{3+\sqrt{3}}{6}, 1].$

 \noindent Let $f_{s_r} \in \mathcal{U}_\infty$ be given by the proper scaling data $s_r : \mathbb{N} \rightarrow \Delta^3 $ and define $$ {\tilde{I}}_{1,r}^n  = [\tilde{b}_c(1),min\{\tilde{b}_c^{-1}( z'_{n})\}] = [f_{s_r}(1), min\{f_{s_r}^{-1}( z'_{n})\}] ,$$  where $ \tilde{b}_c^{-1}(x)$ denotes the preimage(s) of $x$ under $\tilde{b}_c$ and $$ {\hat{I}}_{1,r}^n  = [\tilde{b}_c(y'_{n}),1] = [f_{s,r}(y'_{n}),1].$$  Let $$ h_{s_r,n} : [\tilde{b}_c(1),1] \rightarrow [\tilde{b}_c(1),1] $$ be defined by $$h_{s_r,n} = F_{1,r}(1)\circ F_{1,r}(2)\circ F_{1,r}(3) \circ.....\circ F_{1,r}(n). $$
Furthermore, let $${\tilde{h}}_{s_r,n} : [\tilde{b}_c(1),1] \rightarrow {\tilde{I}}_{1,r}^n \;\; \textrm{and} \;\; {\hat{h}}_{s_r,n} : [\tilde{b}_c(1),1] \rightarrow {\hat{I}}_{1,r}^n $$ be the affine orientation preserving homeomorphisms. Then define $$R_n^rf_s: h_{s_r,n}^{-1}(D_{s_r}) \rightarrow [\tilde{b}_c(1),1]$$  by $$R_n^rf_{s_r}(x) =  \left\{\begin{array}{ll}
R_n^{r-}f_{s_r}(x), \;\;\; \;  \; \textrm{if} \; x \in h_{s_r,n}^{-1}(I_{0,r}^n) \\ R_n^{r+}f_{s_r}(x), \;\;\;  \;  \; \textrm{if} \; x \in h_{s_r,n}^{-1}(I_{2,r}^n)
\end{array}
\right.$$ \\
where, $$R_n^{r-}f_{s_r}: h_{s_r,n}^{-1}(\mathop{\cup}\limits_{n \geq 1}I_{0,r}^n) \rightarrow [\tilde{b}_c(1),1] $$\;\; and $$ R_n^{r+}f_{s_r}: h_{s_r,n}^{-1}(\mathop{\cup}\limits_{n \geq 1}I_{2,r}^n) \rightarrow [\tilde{b}_c(1),1]$$ are defined by $$R_n^{r-}f_{s_r}(x) = {\tilde{h}}_{s_r,n}^{-1} \circ f_{s_r}^{-1} \circ h_{s_r,n}(x)$$ $$R_n^{r+}f_{s_r}(x) = {\hat{h}}_{s_r,n}^{-1} \circ f_{s_r} \circ h_{s_r,n}(x), $$ which are illustrated in Figure~\ref{fig:renomopr}.

\begin{figure}[!htb]
\centering
{\includegraphics [width=110mm]{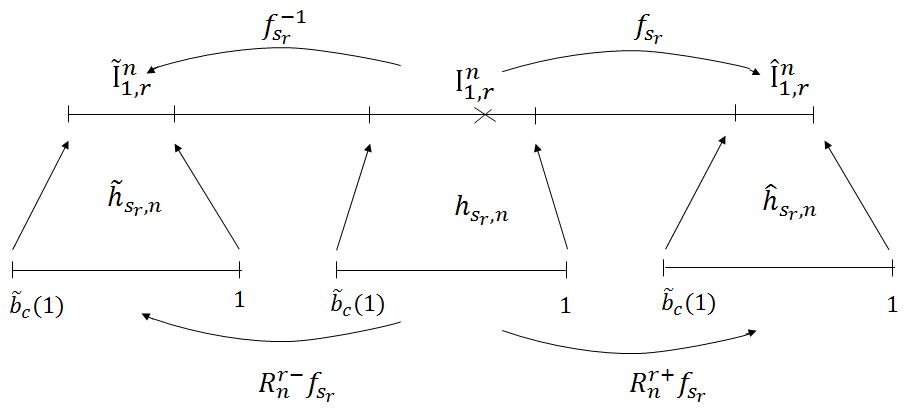}}
\caption{ }
 \label{fig:renomopr}
\end{figure}

\noindent Let $\sigma: (\Delta^3)^{\mathbb{N}} \rightarrow (\Delta^3)^{\mathbb{N}}$ be the shift map which is defined as $ \sigma(s_r^1s_r^2s_r^3s_r^4....) = (s_r^2s_r^3s_r^4....), $ where $s_r^i \in \Delta^3$ for all $i \in \mathbb{N}.$

\begin{Lem}\label{lem1r}
\ Let $s_r : \mathbb{N} \rightarrow \Delta^3$ be proper scaling data such that $f_{s_r}$ is infinitely renormalizable. Then $$R_n^r{f_{s_r}} = f_{\sigma^{n}(s_r)}. $$
\end{Lem}

\noindent Let $f_{s_r}$ be infinitely renormalization, then for $n \geq 0,$ we have $$ f_{s_r}^{3^n}: D_{s_r} \cap I_{1,r}^n \rightarrow I_{1,r}^n $$ is well defined. \\
\ Define the renormalization $R^r : \mathcal{U}_\infty \rightarrow \mathcal{U}_\infty$ by
$$ R^rf_{s_r} = h_{s_r,1}^{-1} \circ f_{s_r}^3 \circ h_{s_r,1} .$$

\noindent The maps $f_{s_r}^{3^n-2} : {\tilde{I}}_{1,r}^n \rightarrow I_{1,r}^n $ and $f_{s_r}^{3^n-1} : {\hat{I}}_{1,r}^n \rightarrow I_{1,r}^n $ are the affine homeomorphisms whenever $f_{s_r} \in \mathcal{U}_\infty$. Then

\begin{Lem}\label{lem2r}
\ We have $(R^r)^n{f_{s_r}} : D_{\sigma^{n}(s_r)} \rightarrow [\tilde{b}_c(1),1]$ and $(R^r)^n{f_{s_r}} = R_n^r{f_{s_r}}.$
\end{Lem}

\noindent The~\cref{lem1} and~\cref{lem2}  give the following result.

\begin{prop}\label{pr}
\ There exists a map $f_{s_r^*} \in  \mathcal{U}_\infty,$ where $s_r^*$ is characterized by $$R^rf_{s_r^*} = f_{s_r^*}. $$
\end{prop}
\begin{proof}
 Consider $s_r : \mathbb{N} \rightarrow \Delta^3$ be proper scaling data such that $f_{s_r}  $ is an infinitely renormalizable. Let $c_n$ be the critical point of $f_{\sigma^n(s_r)}.$ Then

\FloatBarrier
\begin{figure}[h]
\centering
{\includegraphics [scale=0.5]{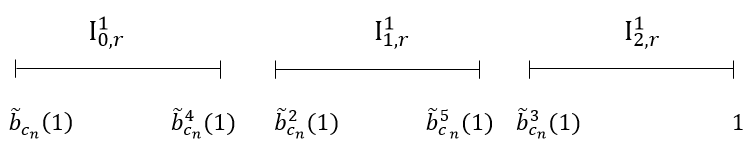}}
\caption{Length of intervals}
\label{fig:ilengthr}
\end{figure} 
\FloatBarrier

\noindent from Figure~\ref{fig:ilengthr}, we have the following scaling ratios
\begin{eqnarray}
 s_{0,r}(n) &= \frac{\tilde{b}_{c_n}^4(1)-\tilde{b}_{c_n}(1)}{1-\tilde{b}_{c_n}(1)} \label{eq11} \\ 
  s_{1,r}(n) &= \frac{\tilde{b}_{c_n}^5(1)-\tilde{b}_{c_n}^2(1)}{1-\tilde{b}_{c_n}(1)} \label{eq12} \\
  s_{2,r}(n) &= \frac{1-\tilde{b}_{c_n}^3(1)}{1-\tilde{b}_{c_n}(1)} \label{eq13} \\
 c_{n+1} &= 1-\frac{c_n-\tilde{b}_{c_n}^2(1)}{s_{1,r}(n)}  \equiv \mathcal{R}(c_n).  \label{eq14}
\end{eqnarray}

\noindent Since $(s_{0,l}(n), s_{1,l}(n), s_{2,l}(n)) \in \Delta^3,$ this implies the following conditions 
\begin{eqnarray}
 s_{0,r}(n),\; s_{1,r}(n), \; s_{2,r}(n) &> 0 \label{eq15} \\
 s_{0,r}(n) +s_{1,r}(n) +s_{2,r}(n) & < 1 \label{eq16} 
 \end{eqnarray}

\ As the intervals $I_{i,r}^n,$ for $i = 0,1,2,$ are mutually disjoint, we will introduce the gap ratios $g_{0,r}^n$ and $g_{1,r}^n$ in between $I_{0,r}^n \;\&\; I_{1,r}^n$ and $I_{1,r}^n \; \& \; I_{2,r}^n$ respectively. The gap ratios are defined as,  \\
for $n \in \mathbb{N},$
\begin{eqnarray}
g_{0,r}^n = \frac{\tilde{b}_{c_n}^2(1)-\tilde{b}_{c_n}^4(1)}{1-\tilde{b}_{c_n}(1)} \equiv G_{0,r}(c_n) >0  \label{gr1} \\
g_{1,r}^n = \frac{\tilde{b}_{c_n}^3(1)-\tilde{b}_{c_n}^5(1)}{1-\tilde{b}_{c_n}(1)} \equiv G_{1,r}(c_n) > 0 \label{gr2} \\
\frac{3+\sqrt{3}}{6} < c_n < 1 \label{eq17}
\end{eqnarray}

\noindent We use Mathematica for solving the equations~(\ref{eq11}),~(\ref{eq12}) and~(\ref{eq13}), we have the expressions for $s_{0,r}(n),\;s_{1,r}(n)$ and $\;s_{2,r}(n).$  Let $s_{i,r}(n) \equiv S_{i,r}(c_n)$ for $i = 0,1,2.$  


\noindent Note that the conditions (\ref{eq15}), (\ref{gr1}) and (\ref{gr2}) give the condition (\ref{eq16}) $$ 0<\sum_{i=0}^{2} s_{i,r}(n) < 1.$$
The conditions ~(\ref{eq15}) together with ~(\ref{gr1}) to ~(\ref{eq17}) define the feasible domain $F_d^r$ is to be:  
\begin{eqnarray}\label{eqfdr}
\fl F_d^r = \Big\{ \; c \in \left(\frac{3+\sqrt{3}}{6},\; 1\right) \; : \;  S_{i,r}(c) > 0 \; \textrm{for} \; i=0,1,2, & G_{0,r}(c) >0,  G_{1,r}(c) >0  \Big\} 
\end{eqnarray} 

\noindent One can compute feasible domain $F_d^r$ as described in subsection \ref{p21}.
This yields:
\begin{eqnarray*}
	F_d^r = (0.800587...,\; 0.805729...) \cup (0.805729...,\;  0.811184...) \equiv F_{d_1}^r \cup F_{d_2}^r.
\end{eqnarray*}

\noindent From the Eqn.(\ref{eq04}), the graphs of $\mathcal{R}(c)$ are plotted in the sub-domains $F_{d_1}^r$ and $F_{d_2}^r$ of $F_d^r$ which are shown in Figure \ref{fig:Rr}.

\begin{figure}[htp]
	\centering
	\begin{subfigure}[b]{0.5\linewidth}
		\centering\includegraphics[width=160pt]{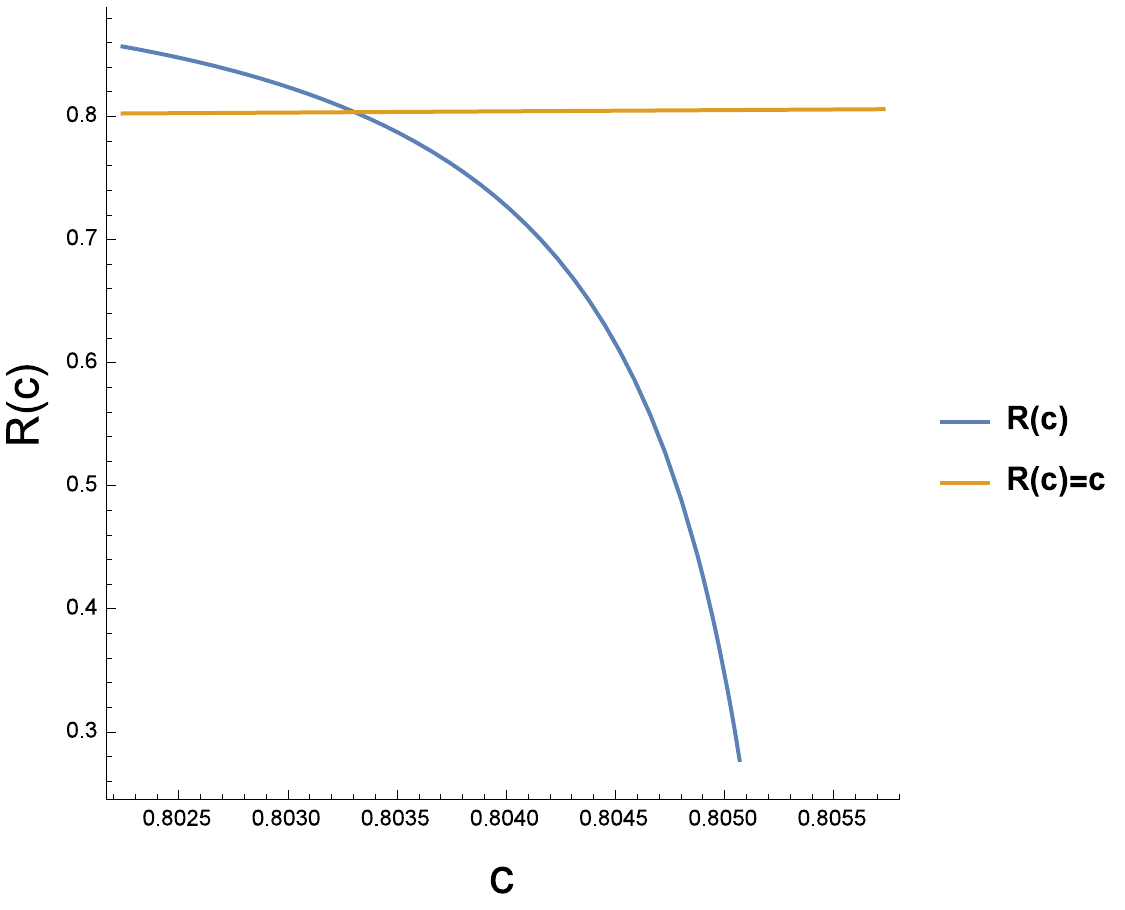}
		\subcaption{ $\mathcal{R}$ has only one fixed point in $F_{d_1}^r.$ }
		\label{fig:picf1r}
	\end{subfigure}%
	\begin{subfigure}[b]{0.5\linewidth}
		\centering\includegraphics[width=160pt]{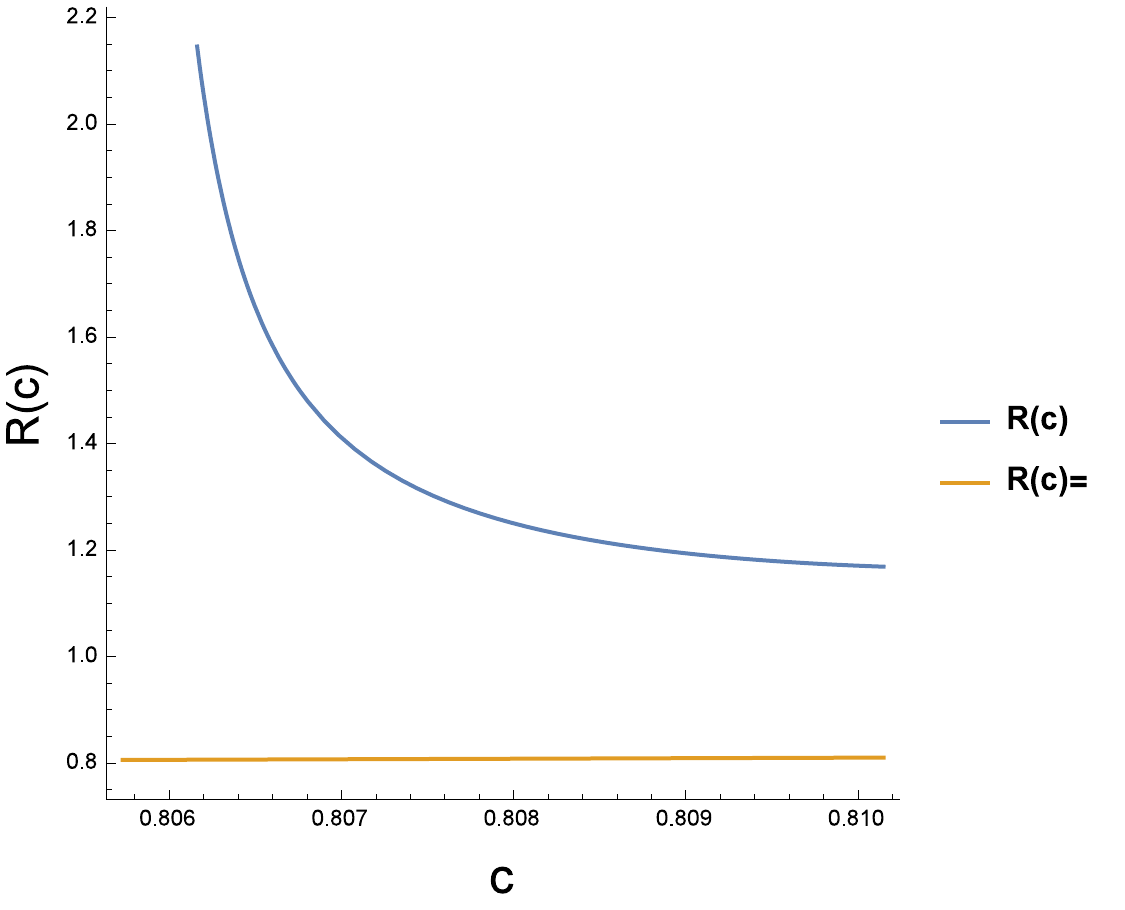}
		\subcaption{$\mathcal{R}$ has no fixed point in $F_{d_2}^r.$ }
		\label{fig:picf2r}
	\end{subfigure}
	\caption{The graph of $\mathcal{R} : F_d^r \rightarrow \mathbb{R}$ and the diagonal $\mathcal{R}(c) = c.$}
	\label{fig:Rr}
\end{figure} 

The map $\mathcal{R}  : F_d^r \rightarrow \mathbb{R}$ is expanding in the neighborhood of fixed point $c_r^*$ which is illustrated in Figure~\ref{fig:picf1r}. By Mathematica computations, we get an unstable fixed points $c_r^* = 0.803307...$ in $F_d^r $ such that \\ $$\mathcal{R}(c_r^*) = c_r^*$$ corresponds to an infinitely renormalizable maps $f_{s_r^*}.$ We observe that the map $ f_{{s_r}^*}$ corresponding to $c_r^*$ has the following property $$\{c_r^*\} = \bigcap\limits_{n \geq 1} I_{1,r}^n. $$

\noindent In other words, consider the scaling data ${s_r}^* : \mathbb{N} \rightarrow \Delta^3$  with 
 \begin{eqnarray*}
 {s_r}^*(n) &= (s_{0,r}^*(n), s_{1,r}^*(n), s_{2,r}^*(n) )  \\
  &=\left(\frac{\tilde{b}_{c_r^*}^4(1)-\tilde{b}_{c_r^*}(1)}{1-\tilde{b}_{c_r^*}(1)}, \; \frac{\tilde{b}_{c_r^*}^5(1)-\tilde{b}_{c_r^*}^2(1)}{1-\tilde{b}_{c_r^*}(1)} , \; \frac{1-\tilde{b}_{c_r^*}^3(1)}{1-\tilde{b}_{c_r^*}(1)}\right).
 \end{eqnarray*}
  Then $\sigma({s_r^*} ) = s_r^* $ and using Lemma \ref{lem1r} we have $$R^rf_{s_r^*} = f_{s_r^*}.$$
\end{proof}

\noindent For a given pair of proper scaling data $s = (s_l, s_r),$ we defined a map
 $$ f_{s} : D_{s_l} \cup D_{s_r} \rightarrow [0,1] $$   as        $$f_s(x) = \left\{\begin{array}{ll}
f_{s_l}(x), \;\;\; \;  \; \textrm{if} \; x \in D_{s_l}  \\ f_{s_r}(x), \;\;\;  \;  \; \textrm{if} \; x \in D_{s_r}
\end{array}
\right.$$
 Then, the renormalization of $f_s$ is defined as 
$$Rf_{s}(x) = \left\{\begin{array}{ll}
R^lf_{s_l}(x), \;\;\; \;  \; \textrm{if} \; x \in D_{s_l}  \\ R^rf_{s_r}(x), \;\;\;  \;  \; \textrm{if} \; x \in D_{s_r}
\end{array}
\right.$$
\ From proposition \ref{pl} and \ref{pr}, we conclude that $f_{s_l^*}$ and $f_{s_r^*}$ are period tripling infinitely renormalizable maps corresponding to the proper scaling data ${s_l^*}$ and ${s_r^*},$ respectively. Then, for a given pair of scaling data $s^* = (s_l^*, s_r^*),$ we have
$$Rf_{s^*}(x) = \left\{\begin{array}{ll}
R^lf_{s_l^*}(x), \;\;\; \;  \; \textrm{if} \; x \in D_{s_l^*}  \\ R^rf_{s_r^*}(x), \;\;\;  \;  \; \textrm{if} \; x \in D_{s_r^*}
\end{array} 
\right.$$ $$ 
 \hspace{1cm}= \left\{\begin{array}{ll}
f_{s_l^*}(x), \;\;\; \;  \; \textrm{if} \; x \in D_{s_l^*}  \\ f_{s_r^*}(x), \;\;\;  \;  \; \textrm{if} \; x \in D_{s_r^*}
\end{array}
\right.  $$
$$ \hspace{-2.3cm} =f_{s^*}(x)$$
The above construction will lead to the following theorem,

\begin{thm} \label{thm1}
\ There exists a map $f_{s^*} \in  B_\infty,$ where $s^* = (s_l^*, s_r^*) $ is characterized by $$Rf_{s^*} = f_{s^*}. $$ In particular, $ B_\infty  = \{f_{s^*}\}.$
\end{thm}

\begin{rem}\label{rm1}
	If $f_{s^*}$ is the map with a pair of proper scaling data $s^* = (s_l^*, s_r^*) $ then the scaling data holds the following properties,  
	\begin{enumerate}
		\item[(i)] $ s_{2,l}^*  \leq (s_{1,l}^*)^2$
		\item[(ii)] $ s_{2,r}^*  \leq (s_{1,r}^*)^2$
	\end{enumerate}
\end{rem}

\begin{rem}
The invariant Cantor set of the map $f_{s^*}$ is next in complexity to the invariant doubling Cantor set of piece-wise affine renormalizable map \cite{CMMT} in the following sense,
\begin{enumerate}
	\item[(i)] like the both Cantor set, on each scale and everywhere the same scaling ratio are used.
	\item[(ii)] But unlike the doubling Cantor set, there are now the pair of three ratios at each scale.
\end{enumerate}
\noindent Furthermore, the geometry of the invariant Cantor set of $f_{s^*}$ is different from the geometry of the invariant Cantor set of piece-wise affine period tripling renormalizable map because the Cantor set of $f_{s^*}$ has $2-$copy of Cantor set of \cite{RK}. 	
\end{rem}

\section{$C^{1+Lip}$ extension of $f_{s^*} $} \label{extsn}
In Section \ref{p2}, we have constructed a piece-wise affine infinitely renormalizable map $f_{s^*}$ corresponding to the pair of scaling data $s^* = (s_l^*,\; s_r^*).$ Let us define a pair of scaling functions $$S_l: [0,\;b_{c_l^*}(0)]^2 \rightarrow [0,\;b_{c_l^*}(0)]^2 $$ $$S_r: [\tilde{b}_{c_r^*}(1),\;1]^2 \rightarrow [\tilde{b}_{c_r^*}(1),\;1]^2 $$ as 
\begin{eqnarray*}
 \hspace{-2cm} S_l\left(\begin{array}{ll}
x \\ y
\end{array} 
\right) =  \left(\begin{array}{ll}
b_{c_l^*}^2(0)-s_{1,l}^* \cdot x \\ s_{2,l}^*\cdot y
\end{array} 
\right);
& \;\;\;\;\;\;\;\; S_r\left(\begin{array}{ll}
x \\ y
\end{array} 
\right) =  \left(\begin{array}{ll}
\tilde{b}_{c_r^*}^2(1)+s_{1,r}^* \cdot (1-x) \\ 1-s_{2,r}^*\cdot (1-y)
\end{array} 
\right). \hspace{-1.5cm}
\end{eqnarray*}

\vspace{0.5cm}

Let $G$ be the graph of $g_{s^*}$ which is an extension of $f_{s^*}$ where $f_{s^*} : D_{s_l^*}\cup  D_{s_r^*} \rightarrow [0,\;1].$ Let $G_l^1$ and $G_l^2$ are the graphs of $g_{s^*}|_{[y_1,\;z_0]}$ and $g_{s^*}|_{[y_0,\;z_1]}$ respectively. Also, $G_r^1$ and $G_r^2$ are the graphs of $g_{s^*}|_{[z_0',\;y_1']}$ and $g_{s^*}|_{[z_1',\;y_0']}$ respectively which are shown in Figure \ref{fig:lipextsn}. Also, note that $G_r^1$ and $G_r^2$ are the reflections of $G_l^1$ and $G_l^2$ across the point $\left(\frac{1}{2},\frac{1}{2}\right)$ respectively. Define
$$ G_l = \cup_{n \geq 1} S_l^n(G_l^1 \cup G_l^2) \;\; \textrm{and} \;\;   G_r =  \cup_{n \geq 1} S_r^n(G_r^1 \cup G_r^2).$$
 Then, $G_l$ is the graph of a unimodal map $g_{s_l^*}$ which extends $f_{s_l^*}$ and $G_r$ is the graph of a unimodal map $g_{s_r^*}$ which extends $f_{s_r^*}$. Consequently,  $ G$ is the graph of $g_{s^*} = g_{s_l^*} \oplus g_{s_r^*}.$ We claim that $g_{s^*}$ is a $C^{1+Lip}$ bimodal map.
  
\noindent Let $B_l^0 = [0,\;b_{c_l^*}(0)] \times [0,\;b_{c_l^*}(0)]$ and $B_r^0 = [\tilde{b}_{c_r^*}(1),\; 1] \times [\tilde{b}_{c_r^*}(1),\;1].$ \\ For $n \in \mathbb{N},$ define $$B_l^n = S_l^n(B_l^0) \;\;\;\;\;\;\; \textrm{and} \;\;\;\;\;\;\; B_r^n = S_r^n(B_r^0)$$ as\\
 $ B_l^n=  \left\{\begin{array}{ll}
 \;[z_n,\;y_n] \times [0,\; \hat{y}_{n}], \;\;\; \textrm{if}\; n \; \textrm{is odd}  \\ \ [y_n,\;z_n] \times [0,\;\hat{y}_{n}], \;\;\; \textrm{if}\; n  \; \textrm{is even}
\end{array}
\right.$
\\ and \\
$ B_r^n=  \left\{\begin{array}{ll}
 \;[y_n',\;z_n'] \times [\hat{y'}_{n},\;1], \;\;\; \textrm{if}\; n \; \textrm{is odd}  \\ \ [z_n',\;y_n'] \times [\hat{y'}_{n},\;1], \;\;\; \textrm{if}\; n  \; \textrm{is even}.
\end{array}
\right.$ \\
\noindent Let $p_l^n$ and $p_r^n$ be the points on the graph of the bimodal map $b_{{c}_l^*}(x)$ and $b_{{c}_r^*}(x)$ respectively. For all $n \in \mathbb{N},$ $p_l^n$ and $p_r^n$ are defined as \\
$ p_l^n=  \left\{\begin{array}{ll}
\left(y_{\frac{n+1}{2}},\hat{y}_{\frac{n+1}{2}}\right), \;\;\; \textrm{if}\; n \; \textrm{is odd}  \\ \left(z_{\frac{n}{2}},\hat{z}_{\frac{n}{2}}\right), \;\;\; \;\;\;\; \;\;\textrm{if}\; n  \; \textrm{is even}
\end{array}
\right.$\\ 
$ p_r^n=  \left\{\begin{array}{ll}
\left(y'_{\frac{n+1}{2}},\hat{y'}_{\frac{n+1}{2}}\right), \;\;\; \textrm{if}\; n \; \textrm{is odd}  \\ \left(z'_{\frac{n}{2}},\hat{z'}_{\frac{n}{2}}\right), \;\;\; \;\;\;\; \;\;\textrm{if}\; n  \; \textrm{is even}
\end{array}
\right.$ \\
where  $ \hat{y}_{n} = b_{c_l^*}(y_n),$ $\hat{z}_{n} = b_{c_l^*}(z_n),$ $ \hat{y'}_{n} = \tilde{b}_{c_r^*}(y'_n)$ and $\hat{z'}_{n} = \tilde{b}_{c_r^*}(z'_n).$

\FloatBarrier
\begin{figure}[!htb]
\centering
{\includegraphics [width=100mm]{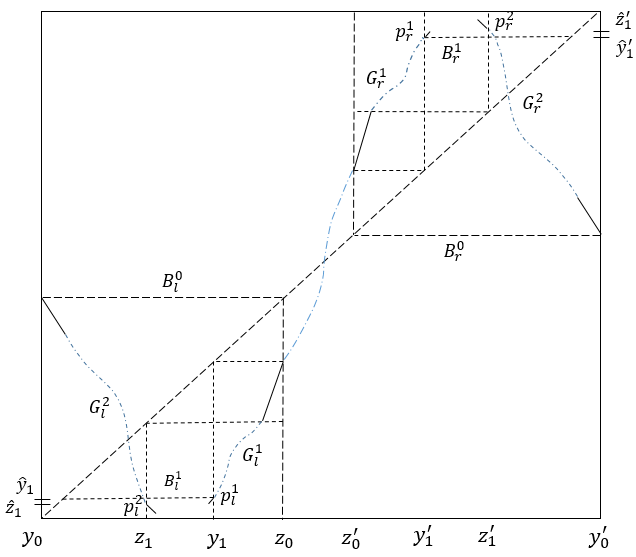}}
\caption{Extension of $f_{s^*}$}
\label{fig:lipextsn}
\end{figure} 
\FloatBarrier

Then the above construction will lead to following proposition,

\begin{prop}
\ $G $ is the graph of $g_{s^*}$ which is a $C^1$ extension of $f_{s^*}.$
\end{prop}
\begin{proof}
\ Since $G_l^1$ and $G_l^2$ are the graph of $f_{s_l^*}|_{[y_1, z_{0}]}$ and $f_{s_l^*}|_{[y_0, z_{1}]},$ respectively, and $G_r^1$ and $G_r^2$ are the graph of $f_{s_r^*}|_{[z_0', y_1']}$ and $f_{s_r^*}|_{[z_1', y_0']},$ respectively, we obtain $G_l^{2n+1} = S_l^n (G_l^1) $ and $G_l^{2n+2} = S_l^n (G_l^2) $ for each $n \in \mathbb{N}.$   Note that $G_l^n$ is the graph of a $C^1$ function defined 
\begin{eqnarray*}
 \textrm{on} \;\; &[z_{\frac{n-1}{2}},\; y_{\frac{n+1}{2}}] \;\;\; \textrm{if} \;\; n \in 4\mathbb{N}-1, \\
 \textrm{on}\;\; &[z_{\frac{n}{2}},\; y_{\frac{n}{2}-1}] \;\;\;\;\;\textrm{if} \;\; n \in 4\mathbb{N}, \\
 \textrm{on}\;\; &[y_{\frac{n+1}{2}},\; z_{\frac{n-1}{2}}] \;\;\;\textrm{if}\;\; n \in 4\mathbb{N}+1, \\
 \textrm{and on}\;\; &[y_{\frac{n}{2}-1},\; z_{\frac{n}{2}}]\;\;\;\;\; \textrm{if}\;\; n \in 4\mathbb{N}+2.
\end{eqnarray*}
\noindent Also, we have $G_r^{2n+1} = S_r^n (G_r^1) $ and $G_r^{2n+2} = S_r^n (G_r^2) $ for each $n \in \mathbb{N}.$ Note that $G_r^n$ is the graph of a $C^1$ function defined 
\begin{eqnarray*}
 \textrm{on} \;\; &[y_{\frac{n+1}{2}}',\; z_{\frac{n-1}{2}}'] \;\;\; \textrm{if} \;\; n \in 4\mathbb{N}-1, \\
 \textrm{on}\;\; &[y_{\frac{n}{2}-1}',\; z_{\frac{n}{2}}'] \;\;\;\;\;\textrm{if} \;\; n \in 4\mathbb{N}, \\
 \textrm{on}\;\; &[z_{\frac{n-1}{2}}',\; y_{\frac{n+1}{2}}'] \;\;\;\textrm{if}\;\; n \in 4\mathbb{N}+1, \\
 \textrm{and on}\;\; &[z_{\frac{n}{2}}',\; y_{\frac{n}{2}-1}']\;\;\;\;\; \textrm{if}\;\; n \in 4\mathbb{N}+2.
\end{eqnarray*}
 To prove the proposition, we have to check continuous differentiability at the points $p_l^n$ and $p_r^n.$ Consider the neighborhoods $(y_1-\epsilon,y_1+\epsilon)$ around $y_1$ and $(z_1-\epsilon,z_1+\epsilon)$ around $z_1$, the slopes are given by an affine pieces of   $f_{s_l^*}$ on the subintervals $(y_1-\epsilon,y_1)$ and $(z_1,z_1+\epsilon)$ and the slopes are given by the chosen $C^1$ extension on $(y_1,y_1+\epsilon)$ and  $(z_1-\epsilon,z_1).$ This implies, $G_l^1$ and $G_l^2$ are $C^1$ at $p_l^1$ and $p_l^2,$ respectively. \\ Let $\gamma_1 \subset G_l $ be the graph over the interval $(y_1-\epsilon,y_1+\epsilon)$ and $\gamma_2 \subset G_l $ be the graph over the interval $(z_1-\epsilon,z_1+\epsilon),$ \\ then the graph $G_l$ locally around $p_l^n$ is equal to $\left\{\begin{array}{ll}
S_l^{\frac{n-1}{2}}{(\gamma_1)} \;\;\; \textrm{if}\; n \; \textrm{is odd} \\ S_l^{\frac{n-2}{2}}{(\gamma_2)} \;\;\; \textrm{if}\; n  \; \textrm{is even}
\end{array}
\right.$. \\ This implies, for $n \in \mathbb{N},$  $G_l^{2n-1}$ is $C^1$ at $p_l^{2n-1}$ and $G_l^{2n}$ is $C^1$ at $p_l^{2n}.$ \\ Hence $G_l$ is a graph of a $C^1$ function on $[0,b_{c_l^*}(0)] \setminus \{c_l^*\}.$\\ We note that the horizontal contraction of $S_l$ is smaller than the vertical contraction. This implies that the slope of $G_l^n$ tends to zero when $n$ is large. Therefore, $G_l$ is the graph of a $C^1$ function  $g_{s_l^*}$ on $[0,b_{c_l^*}].$ 
\noindent In similar way, one can prove that $G_r$ is the graph of a $C^1$ function  $g_{s_r^*}$ on $[\tilde{b}_{c_r^*},1].$ Therefore, $G = G_l \oplus G_r$ is the graph of a $C^1$ bimodal map $g_{s^*} =g_{s_l^*} \oplus g_{s_r^*}$ which is a $C^1$ extension of $f_{s^*}.$
\end{proof}

\begin{prop}
\ Let $g_{s^*}$ be the function whose graph is $G$ then $g_{s^*}$ is a $C^{1+Lip}$ bimodal map. 
\end{prop}
\begin{proof}
\  As the function $g_{s^*}$ is a $C^1$ extension of $f_{s^*}.$ We have to show that, for $i \in  \{l, \;r\},$ $G_i^n$ is the graph of a $C^{1+Lip}$ function 
 $$ g_{s_i^*}^n: Dom(G_i^n) \rightarrow [0,1] $$ with an uniform Lipschitz bound. \\ That is,
  for $n \geq 1,$ 
  $$Lip((g_{s_i^*}^{n+1})') \leq Lip((g_{s_i^*}^n)')$$ 
\noindent let us assume that $g_{s_l^*}^n$ is $C^{1+Lip}$ with Lipschitz constant $\lambda_n$ for its derivatives. We show that $\lambda_{n+1} \leq \lambda_{n}$. \\
 For given $(u,v)$ on the graph of $g_{s_l^*}^n,$ there is $(\tilde{u},\tilde{v}) = S_l(u,v)$ on the graph of $g_{s_l^*}^{n+1}$, this implies $$g_{s_l^*}^{n+1}(\tilde{u}) = s_{2,l}^*\cdot g_{s_l^*}^n(u)$$
 Since $u = \frac{b_{c_l^*}^2(0)-\tilde{u}}{{s_{1,l}^*}},$ we have 
$$g_{s_l^*}^{n+1}(\tilde{u}) = s_{2,l}^*\cdot g_{s_l^*}^n\left(\frac{b_{c_l^*}^2(0)-\tilde{u}}{{s_{1,l}^*}}\right)$$ Differentiate both sides with respect to $\tilde{u}$, we get
$$\left(g_{s_l^*}^{n+1}\right)'(\tilde{u}) = -\frac{s_{2,l}^*}{s_{1,l}^*} \cdot \left(g_{s_l^*}^n\right)'\left(\frac{b_{c_l^*}^2(0)-\tilde{u}}{{s_{1,l}^*}}\right)$$

\noindent Therefore,
\begin{eqnarray*}\label{eqlp}
\fl \left|\left(g_{s_l^*}^{n+1}\right)'(\tilde{u}_1) - \left(g_{s_l^*}^{n+1}\right)'(\tilde{u}_2)\right| &= \left|\frac{s_{2,l}^*}{s_{1,l}^*}\right|\cdot \left|\left(g_{s_l^*}^n\right)'\left(\frac{b_{c_l^*}^2(0)-\tilde{u}_1}{{s_{1,l}^*}}\right)-\left(g_{s_l^*}^n\right)'\left(\frac{b_{c_l^*}^2(0)-\tilde{u}_2}{{s_{1,l}^*}}\right)\right| \nonumber\\ 
&\leq \frac{s_{2,l}^*}{({s_{1,l}^*})^2} \cdot \lambda\left(g_{s_l^*}^n\right)'|\tilde{u}_1-\tilde{u}_2|
\end{eqnarray*}
 
\noindent From remark \ref{rm1}, we have $({s_{1,l}^*})^2 \geq {s_{2,l}^*}.$ Then,
 $$\lambda(g_{s_l^*}^{n+1})' \leq \lambda(g_{s_l^*}^{n})' \leq \lambda(g_{s_l^*}^{1})'.$$
\noindent Similarly, one can show that $$\lambda(g_{s_r^*}^{n+1})' \leq \lambda(g_{s_r^*}^{n})' \leq \lambda(g_{s_r^*}^{1})'.$$ Therefore, choose $\lambda = max\{ \lambda(g_{s_l^*}^{1})',\; \lambda(g_{s_r^*}^{1})' \}$ is the uniform Lipschitz bound. This completes the proof.   
\end{proof}

\noindent Note that for a given pair of proper scaling data $s^* = (s_l^*,s_r^*),$ the piece-wise affine map $f_{s^*}$ is infinitely renormalizable and $g_{s^*}$ is a $C^{1+Lip}$ extension of $f_{s^*}$. This implies $g_{s^*}$ is also renormalizable map. Further, we observe that $R g_{s^*}$ is an extension of $Rf_{s^*}.$ Therefore $R g_{s^*}$ is renormalizable. Hence, $g_{s^*}$ is infinitely renormalizable map  which is not a $C^2$ map. Then we have the following theorem,

\begin{thm}\label{thm:ext}
\ There exists an infinitely renormalizable $C^{1+Lip}$ bimodal map $g_{s^*}$ such that $$ R g_{s^*} = g_{s^*} .$$
\end{thm}

\section{Topological entropy of renormalization}\label{entropy}
\ In this section, we calculate the topological entropy of the renormalization operator defined on the space of $ C^{1+Lip}$ bimodal maps. \\ Let us consider three pairs of $C^{1+Lip}$ maps $\phi_i : [0,z_1] \cup [y_1,b_{c_l^*}(0)] \rightarrow [0,b_{c_l^*}(0)] $ and  $\psi_i : [\tilde{b}_{c_r^*}(1),y_1'] \cup [z_1',1] \rightarrow [\tilde{b}_{c_r^*}(1),1],$ for $i = 0,1,2,$ which extend $f_{s^*}.$ Because  of symmetricity, $\psi_i(x)  = 1-\phi_i(1-x).$ For a sequence $\alpha = \{\alpha_n\}_{n \geq 1} \in \Sigma_3,$ \\ where $\Sigma_3 =  \{ \{x_n\}_{n \geq 1} : x_n \in \{0,\;1,\;2\} \}$ is called full 3-Shift. \\ Now define $$G_l^n(\alpha) = S_l^n(graph \; \phi_{{\alpha}_n}) \;\;\;\; \textrm{and} \;\;\;\; G_r^n(\alpha) = S_r^n(graph \; \psi_{{\alpha}_n}), $$ we have $$G_l(\alpha) =  \bigcup\limits_{n \geq 1} G_l^{n}(\alpha) \;\;\;\; \textrm{and} \;\;\;\; G_r(\alpha) =  \bigcup\limits_{n \geq 1} G_r^{n}(\alpha).$$ Therefore, we conclude that $G(\alpha) = G_l(\alpha) \oplus  G_r(\alpha)$ is the graph of a $C^{1+Lip}$ bimodal map $b_{\alpha}$ by using the same facts of Section \ref{extsn}. \\ The shift map $\sigma : \Sigma_3 \rightarrow \Sigma_3$ is defined as $$\sigma(\alpha_1 \alpha_2 \alpha_3 \ldots) = (\alpha_2 \alpha_3 \alpha_4 \ldots).$$ 

\begin{prop}
The restricted maps $b_\alpha^3 : [y_1,\;z_1] \rightarrow [y_1,\;z_1] $ and $b_\alpha^3 : [y_1',\;z_1'] \rightarrow [y_1',\;z_1'] $ are the unimodal maps for all $\alpha \in \Sigma_3.$ In particular, $b_\alpha$ is a renormalizable map and $R b_\alpha = b_{\sigma(\alpha)}.$
\end{prop}
\begin{proof}
\ We know that $b_\alpha :  [y_1,\;z_1] \rightarrow I_{2,l}^1 $ is a unimodal and onto, $b_\alpha :  I_{2,l}^1 \rightarrow I_{0,l}^1 $ is onto and affine and also $b_\alpha : I_{0,l}^1 \rightarrow  [y_1,\;z_1]  $ is onto and affine. Therefore $b_\alpha^3$ is a unimodal map on $[y_1,\;z_1]$. Analogously, $b_\alpha^3$ is a unimodal map on $[y_1',\;z_1']$. The above construction implies $$R b_\alpha = b_{\sigma(\alpha)}.$$
\end{proof}

\noindent This gives us the following theorem.

\begin{thm} \label{prep2}
\ The renormalization operator $R$ acting on the space of $C^{1+Lip}$ bimodal maps has unbounded topological entropy.
\end{thm}
\begin{proof}
\ From the above construction, we conclude that $\alpha \longmapsto b_\alpha \in C^{1+Lip}$ is injective.  The domain of $R$ contains two copies, namely $\Lambda_1$ and $\Lambda_2,$  of the full 3-shift. As topological entropy $h_{top}$ is an invariant of topological conjugacy. Hence $h_{top}(\left.R\right\vert_{\Lambda_1 \cup \Lambda_2} ) > \ln 3.$ In fact, if we choose $n$ different pairs of $C^{1+Lip}$ maps, say, $\phi_0,\;\phi_1,\;\phi_2,\ldots \phi_{n-1}$ and $\psi_0,\;\psi_1,\;\psi_2,\ldots \psi_{n-1},$ which extends $f_{s^*},$ then it will be embedded two copies of the full $n-\textrm{shift}$ in the domain of $R$.  Hence, the topological entropy of $R$ on $C^{1+Lip}$  bimodal maps is unbounded. 
\end{proof}

\section{An $\epsilon$ perturbation of the scaling data}\label{pert}
\ In this section, we use an $\epsilon$ perturbation on the construction of the scaling data as presented in Section $\ref{p2},$ to obtain the following theorem

\begin{thm}
\  There exists a continuum of fixed points of the renormalization operator acting on $C^{1+Lip}$ bimodal maps.
\end{thm}
\begin{proof}
 Consider an $\epsilon$ variation on scaling data and we modify the construction which is described in section ~\ref{p2}.
 
 Let us define the neighborhoods $N_{\epsilon}^l$ and $N_{\epsilon}^r$ about the respective points $(b_{c}^3(0),\; b_{c}^4(0))$  and $(b_{c}^3(1),\; b_{c}^4(1))$  as 
 \begin{eqnarray*}
N_{\epsilon}^l(b_{c}^3(0),\; b_{c}^4(0)) = \{ (b_{c}^3(0),\; \epsilon \cdot b_{c}^4(0)) \; : \; \epsilon > 0 \; \textrm{ and} \; \epsilon \; \textrm {close to } 1  \} \\ 
N_{\epsilon}^r(b_{c}^3(1),\; b_{c}^4(1)) = \{ (b_{c}^3(1),\; \epsilon \cdot b_{c}^4(1)) \; : \; \epsilon > 0 \; \textrm{ and} \; \epsilon \; \textrm {close to } 1  \}
 \end{eqnarray*}
\noindent \textbf{Case (i)} The perturbed scaling data on $I_0^l,$ then the scaling ratios are defined as 
\begin{eqnarray*}
s_{2,l}(c,\epsilon) &=   \frac{b_{c}^3(0)}{b_{c}(0)}  \\ 
s_{0,l}(c,\epsilon) &=  \frac{b_{c}(0)-\epsilon b_{c}^4(0)}{b_{c}(0)}  \\
s_{1,l}(c,\epsilon) &=  \frac{b_{c}^2(0)-b_{c}(\epsilon b_{c}^4(0))}{b_{c}(0)},  
\end{eqnarray*} 
where $c \in (0, \frac{3-\sqrt{3}}{6}).$ Also, we define
\begin{eqnarray*}
	\mathcal{R}(c,\epsilon) &= \frac{b_{c}^2(0)-c}{s_{1,l}(c,\epsilon)}.
\end{eqnarray*}
From subsection~\ref{p21}, we know that the map $\mathcal{R}$ which is defined in Eqn. \ref{eq04}, has unique fixed point $c^*$. Consequently, for a given $\epsilon$ close to $1,$ $\mathcal{R}(c,\epsilon)$ has only one unstable fixed point, namely $c_{\epsilon}^*.$ Therefore, we consider the perturbed scaling data ${s}_{l,\epsilon}^* : \mathbb{N} \rightarrow \Delta^3$ with 
\begin{eqnarray*}
	{s}_{l,\epsilon}^* = \left( \frac{b_{c_{\epsilon}^*}(0)-\epsilon b_{c_{\epsilon}^*}^4(0)}{b_{c_{\epsilon}^*}(0)}, \; \frac{b_{c_{\epsilon}^*}^2(0)-b_{c_{\epsilon}^*}(\epsilon b_{c_{\epsilon}^*}^4(0))}{b_{c_{\epsilon}^*}(0)}, \;\frac{b_{c_{\epsilon}^*}^3(0)}{b_{c_{\epsilon}^*}(0)} \right).
\end{eqnarray*}

\noindent Then $\sigma(s_{l,\epsilon}^* ) = {s}_{l,\epsilon}^* $ and using Lemma \ref{lem1}, we have
$$R^lf_{{s}_{l,\epsilon}^*} = f_{{s}_{l,\epsilon}^*}. $$

\noindent \textbf{Case (ii)} Considering the perturbed scaling data on $I_0^r,$ one have the scaling data ${s}_{r,\epsilon}^* : \mathbb{N} \rightarrow \Delta^3$ with 
\begin{eqnarray*}
	{s}_{r,\epsilon}^* = \left( \frac{\epsilon b_{c_{\epsilon}^*}^4(1)-b_{c_{\epsilon}^*}(1)}{1-b_{c_{\epsilon}^*}(1)}, \; \frac{b_{c_{\epsilon}^*}(\epsilon b_{c_{\epsilon}^*}^4(1))-b_{c_{\epsilon}^*}^2(1)}{1-b_{c_{\epsilon}^*}(1)}, \;\frac{1-b_{c_{\epsilon}^*}^3(1)}{1-b_{c_{\epsilon}^*}(1)} \right).
\end{eqnarray*}

\noindent Then $\sigma(s_{r,\epsilon}^* ) = {s}_{r,\epsilon}^* $ and using Lemma \ref{lem1r}, we have
$$R^rf_{{s}_{r,\epsilon}^*} = f_{{s}_{r,\epsilon}^*}. $$
 
Moreover, $f_{{s}_{l,\epsilon}^*}$ and $f_{{s}_{r,\epsilon}^*}$ are the piece-wise affine maps which are infinitely renormalizable. For a given pair of proper scaling data $ {s}_{\epsilon}^* = ({s}_{l,\epsilon}^*,\;{s}_{r,\epsilon}^*),$ we have $$Rf_{{s}_{\epsilon}^*} = f_{{s}_{\epsilon}^*} $$

Now we use similar extension described in section~\ref{extsn}, then we get $g_{{s}_{\epsilon}^*}$ is the $C^{1+Lip}$ extension of $f_{{s}_{\epsilon}^*}$. This implies that $g_{{s}_{\epsilon}^*}$ is a renormalizable map. As $R g_{{s}_{\epsilon}^*}$ is an extension of $Rf_{{s}_{\epsilon}^*}.$ Therefore $R g_{{s}_{\epsilon}^*}$ is renormalizable. Hence, for each $\epsilon$ close to $1,$ $g_{{s}_{\epsilon}^*}$ is a fixed point of the renormalization. This proves the existence of a continuum of fixed points of the renormalization.
\end{proof}

\begin{rem}
	In particular, for two different perturbed scaling data  $s_{\epsilon_1^*}$ and $s_{\epsilon_2^*},$ one can  construct two infinitely renormalizable maps $g_{s_{\epsilon_1^*}}$ and $g_{s_{\epsilon_2^*}}.$ Therefore, the respective Cantor attractors will have different scaling ratios. Consequently, it shows the non-rigidity for low smooth symmetric bimodal maps.
\end{rem}

\section{Conclusions}
\ In  this paper, we have investigated the existence of fixed point of the renormalization operator which is defined on the space of symmetric bimodal maps with low smoothness. For a given pair of proper scaling data $s^* = (s_l^*,\; s_r^*),$ we have first constructed the piece-wise affine infinitely renormalizable map $f_{s^*}$ which is the only fixed point of the renormalization. We observe that the geometry of invariant Cantor set is more complex than the geometry of the Cantor set of piece-wise affine period doubling renormalizable map\cite{CMMT}. Further, we have extended this fixed point $f_{s^*}$ to a $C^{1+Lip}$ symmetric bimodal map. Moreover, we proved that the renormalization operator acting on the space of $C^{1+Lip}$ symmetric bimodal maps has infinite topological entropy. Finally, we proved the existence a continuum of fixed points of renormalization by considering a small perturbation on the scaling data. Consequently, it showed  the non-rigidity of the Cantor attractors of infinitely renormalizable symmetric bimodal maps with low smoothness.


\section*{References}
\begin{thebibliography}
\


\bibitem{Fe}
Feigenbaum M J 1978, 
\newblock { Quantitative universality for a class of non-linear transformations,}
{\it J. Stat. Phys.}, {\bf 19} 25-52.

\bibitem{Fe2}
Feigenbaum M J 1979, 
\newblock{ The universal metric properties of nonlinear transformations,}
{\it  J. Stat. Phys.}, {\bf 21} 669-706.

\bibitem{CT}
 Coullet P, Tresser C 1978,
\newblock { It\'{e}ration d'endomorphisms et groupe de renormalisation,}
{\it  J. Phys. Colloque },  {\bf C5} 25-28.

\bibitem{JD}
Jonkar L, Rand D 1980,
\newblock {Bifurcations in one dimension I. The nonwandering set,}
{\it Invent Math},  {\bf 62} 347–365.

\bibitem{SVST}
Strien S V 1988,
\newblock { Smooth dynamics on the interval, in : New Directions in Dynamical Systems, eds. T. Bedford and J. Swift,}
{\it Cambridge Univ. Press, Cambridge },   57-119.

\bibitem{MCT}
Mackay R S, Tresser C 1986,
\newblock { Transition to topological chaos for circle maps,}
{\it Physica D: Nonlinear Phenomena },  {\bf 19} 206-237.

\bibitem{DVeitch}
Veitch D 1994,
\newblock { Renormalization of $C^{0}$ bimodal maps, }
{\it Physica D}, {\bf 71}  269-284.

\bibitem{DSmania}
Smania D 2005, 
\newblock { Phase space universality for multimodal maps,}
{\it  Bull. Braz. Math. Soc.},  {\bf 36} 225-274.


\bibitem{CMMT}
Chandramouli V V M S, Martens M, Melo W de, Tresser C P 2009,
\newblock { Chaotic period doubling,}
{\it Ergodic Theory and Dynamical Systems}, {\bf 29} 381-418.

\bibitem{RK}
\newblock Kumar R, Chandramouli V V M S 2020,
\newblock {Period tripling and quintupling renormalizations below $C^2$ space}, preprint,
{\it arXiv:2010.01293 [math.DS]}.
 
\bibitem{tresser}
Tresser C 1991,
\newblock { Fine Structure of Universal Cantor Sets,}
 {\it  Instabilities and Nonequilibrium Structures III, E. Tirapegui and W. Zeller Eds., (Kluwer, Dordrecht/Boston/London} . 
{ 27-42.}

\bibitem{WDMVS}
Welington de Melo \& Sebastian van Strien,
\newblock {\it One-Dimensional dynamics,} (Springer Verlag, Berlin; 1993).





%




 



\end{thebibliography}
\end{document}